\documentclass[12pt,english]{article}

\usepackage{amsfonts,amsmath,amsthm,amscd,amssymb,latexsym,cite,verbatim}

\theoremstyle{plain}
\newtheorem{theorem}{Theorem}[section]
\newtheorem{lemma}{Lemma}[section]

\newtheorem{definition}{Definition}[section]
\newtheorem{remark}{Remark}[section]

\newcommand{\keywords}{\textbf{Key words. }\medskip}
\newcommand{\subjclass}{\textbf{MSC 2010. }\medskip}
\renewcommand{\abstract}{\textbf{Abstract. }\medskip}

\numberwithin{equation}{section}
\begin{document}

\title{Large and very singular solutions to semilinear elliptic equations}

\author{Andrey~E.~Shishkov
\\ RUDN University, Russian Federation
}

\date{}

\maketitle

\begin{abstract}
We consider equation $-\Delta u+f(x,u)=0$ in smooth bounded domain $\Omega\in\mathbb{R}^N$, $N\geqslant2$, with $f(x,r)>0$ in $\Omega\times\mathbb{R}^1_+$ and $f(x,r)=0$ on $\partial\Omega$. We find the condition on the order of degeneracy of $f(x,r)$ near $\partial\Omega$, which is a criterion of the existence-nonexistence of a very singular solution with a strong point singularity on $\partial\Omega$. Moreover, we prove that the mentioned condition is a sufficient condition for the uniqueness of a large solution and conjecture that this condition is also a necessary condition of the uniqueness.
\end{abstract}



\section{Introduction and main results}

This paper deals with two problems:

\noindent 1) the uniqueness of large solutions,

\noindent 2) the existence of very singular solutions

\noindent to a semilinear elliptic equation of the form:
\begin{equation}\label{ad1*}
-\Delta u + f(x,u)=0\quad\text{in }\Omega\in\mathbb{R}^N,\ N>1,
\end{equation}
where nonlinear absorption term $f(x,s)>0$ $\forall\,x\in\Omega$, $\forall\,s>0$, degenerates on $\partial\Omega$:
\begin{equation}\label{ad2*}
f(x,s)=0\quad\forall\,x\in\partial\Omega,\qquad f(x,0)=0\quad\forall\,x\in\overline{\Omega}.
\end{equation}
When $f(s)$ is monotonic the existence of the large solution, i.e. a solution of equation \eqref{ad1*} satisfying boundary condition:
\begin{equation}\label{ad3*}
\lim_{d(x)\to0}u(x)=\infty,\quad d(x):=dist(x,\partial\Omega),
\end{equation}
is assosiated  with a well known Keller--Osserman \cite{Kel,Osser} condition on the growth of $f(s)$ as $s\to\infty$. An adaptation of the KO-condition to nonmonotonic $f(s)$ was realized in \cite{DDGR}, to general nonlinearities $f(x,s)$ --- in \cite{L-G1}. A generalization of the KO-condition for higher order semilinear equations and inequalities was introduced in \cite{KSh1}. The uniqueness of large solutions was firstly proved by C.~Loewner and L.~Nirenberg in \cite{LoNi} for smooth domain $\Omega$ and $f(s)=s^p$, $p=\frac{N+2}{N-2}$. The first general result about the uniqueness was obtained by C.~Bandle and M.~Marcus \cite{BaMa} for smooth bounded domain $\Omega$ and $f(s)=s^p$, $p>1$. Asymptotic methods, introduced in \cite{BaMa}, was applied to different classes of nonlinearities $f(x,s)$ by many authors (see \cite{L-GMV} and references therein).

It appears clear that the uniqueness of the large solution mostly depends on the order of degeneracy of nonlinearity $f(x,s)$ on the boundary of $\Omega$. So, in \cite{MV4} the uniqueness was proved for $C^2$--smooth bounded domain when:
\begin{equation*}
f(x,s)\geqslant c_0d(x)^\alpha s^p\quad\forall\,x\in\Omega,\ \forall\,s\geqslant0,\ p>1,\ \alpha>0,\ c_0=const>0,
\end{equation*}
where $d(x)$ is from \eqref{ad3*}. In \cite{L-GMV} authors conjectured the uniqueness under the following condition:
\begin{equation*}
f(x,s)\geqslant c_0\exp\left(-\frac{c_1}{d(x)^\alpha}\right) s^p\quad\forall\,x\in\Omega,\ \forall\,s\geqslant0,\ 0<\alpha<1,\ c_1=const\geqslant0.
\end{equation*}
In the present paper, we prove the validity of this hypothesis. Moreover, we prove even more general results.
Namely, in bounded domain $\Omega\in\mathbb{R}^N$ with $C^2$--smooth boundary $\partial\Omega$ we consider the following semilinear equation:
\begin{equation}\label{1.17}
Lu + H(x)u^p :=-\sum_{i,j=1}^{N}\left(a_{ij}(x)u_{x_i}\right)_{x_j} +H(x)u^p=0\quad\text{in }\Omega,\ p>1,
\end{equation}
where $C^{1,\lambda}$--smooth functions $a_{ij}(\cdot)$ satisfy the ellipticity condition:
\begin{equation}\label{1.19}
d_1\mid\xi\mid^2\geqslant\sum_{i,j=1}^Na_{ij}(x)\xi_i\xi_j\geqslant d_0\mid\xi\mid^2 \ \ \forall\,\xi\in\mathbb{R}^N,\ \forall\,x\in\overline{\Omega},\ d_1<\infty,\ d_0>0,
\end{equation}
and the absorption potential $H(\cdot)$ satisfies
\begin{equation}\label{1.20}
H(x)\geqslant h_\omega(d(x))\quad\forall\,x\in\overline{\Omega},\quad h_\omega(s)=:\exp\left(-\frac{\omega(s)}{s}\right) \quad\forall\,s\in(0,\rho_0).
\end{equation}

\begin{theorem}\label{Th.3}
Let potential $H(x)$ satisfy estimate \eqref{1.20}, where nondecreasing continuous function $\omega(\cdot)$ satisfies the technical condition:
\begin{equation}\label{1.23}
s^{\gamma_1}\leqslant\omega(s)<\omega_0=const<\infty \quad\,s\in(0,\rho_0),\ 0<\gamma_1<1
\end{equation}
and the Dini condition
\begin{equation}\label{1.24}
\int_{0}^{c}\frac{\omega(s)}{s}ds<\infty;
\end{equation}
Then equation \eqref{1.17} admits only one large solution in the mentioned domain $\Omega$.
\end{theorem}

%\begin{remark}\label{Rem10}
We conjecture that the Dini condition \eqref{1.24} is also a necessary condition for the uniqueness of the large solution.
%\end{remark}
As an indirect confirmation of the validity of this conjecture we consider the second main result of this paper about the necessity of the Dini condition \eqref{1.24} for the existence of a very singular (v.s.) solution to equations of structure \eqref{1.17}. Let us remind that a v.s. solution was first discovered as a nonnegative solution of the semilinear parabolic equation:
\begin{equation}\label{ad5*}
u_t-\Delta u+hu^p=0\quad\text{in } \mathbb{R}^1_+\times\mathbb{R}^N,\ 1<p<1+2N^{-1},\ h=const>0,
\end{equation}
satisfying the initial condition:
\begin{equation}\label{ad6*}
u(0,x)=\infty\,\delta(x),\quad \delta(x)\text{ is Dirac measure},
\end{equation}
in the following sense: $u(0,x)=0$ $\forall\,x:\mid x\mid\neq0$, and
\begin{equation*}
\lim_{t\to0}\int_{R^N}u_\infty(t,x)dx=\infty.
\end{equation*}
Moreover, v.s. solution $u_\infty(t,x)$ can be obtained as $\lim_{k\to\infty}u_k(t,x)$, where $u_k(t,x)$ is a solution of \eqref{ad5*} satisfying the initial condition $u_k(0,x)=k\delta(x)$ (see \cite{BPT,KP1,EK} and references therein). The next step was to study v.s. solutions to problem \eqref{ad5*}, \eqref{ad6*} with variable absorption potential $h=h(t,x)\geqslant0$, degenerating on initial hyperplane:
\begin{equation}\label{ad7*}
h(0,x)=0\quad\forall\,x\in\mathbb{R}^N.
\end{equation}
A new phenomenon was observed in \cite{MV5}: if $h=h(t)=\exp\left(-\omega t^{-1}\right)$, $\omega=const>0$, then $u_\infty(t,x)=\lim_{k\to\infty}u_k(t,x)$ is not a v.s. solution, but is a large solution, namely, $u_\infty(0,x)=\infty$ $\forall\,x\in\mathbb{R}^N$. In \cite{ShV2} it was found a sharp condition on the degeneracy of $h(t,x)$ which guarantees the existence of v.s. solution $u_\infty(t,x)$ with strong point singularity: $h(t)\geqslant\exp\left(-\omega(t)t^{-1}\right)$, where $\omega(\cdot)$ is a continuous nondecreasing function, satisfying the following Dini condition:
\begin{equation}\label{ad8*}
\int_0^cs^{-1}\omega(s)^\frac12<\infty.
\end{equation}
So far, we haven't known whether this condition is also a necessary condition for the existence of a v.s. solution to semilinear parabolic equations of the structure \eqref{ad5*}. But in the case of the semilinear elliptic equation \eqref{1.17} the role of the Dini-type condition \eqref{1.24} for the existence of the corresponding v.s. solution has been studied more fully by now.
Particularly, in \cite{ShV1} the following result about the sufficiency was proved. Let $\{u_k(x)\}$ be a sequence of solutions of equation \eqref{1.17}, \eqref{1.19}, \eqref{1.20}, satisfying the boundary condition:
\begin{equation}\label{ad9*}
u_k=k\delta_a(x),\quad\text{on }\partial\Omega,\ a\in\partial\Omega,\ k=1,2,...
\end{equation}
Let the potential $H(x)$ satisfy estimate \eqref{1.20}, where nonnegative function $\omega(s)$ satisfies all conditions of Theorem~\ref{Th.3}. Then $u_\infty(x)=\lim_{k\to\infty}u_k(x)$ is a v.s. solution of \eqref{1.17}, i.e. a solution with a strong (more strong than the corresponding Poisson kernel) boundary singularity at $a\in\partial\Omega$ and $\lim_{x\to y}u(x)=0$ $\forall\,y\in\partial\Omega\setminus\{a\}$.

%Second main result of present paper is the following statement about necessity of Dini condition for existence of v.s. solution.

Let us consider the following model problem:
\begin{equation}\label{1.1}
-\Delta u + h_\omega(\mid x'\mid )u^p=0\quad\text{in }\Omega:=\mathbb{R}^N_+=\{x\in\mathbb{R}^N, x_N>0\},
\end{equation}
\begin{equation}\label{1.2}
u\mid_{x_N=0}=K\delta_a(x),\qquad a\in L\subset\partial\Omega,\ K\in\,\mathbb{R}_+^1,%:=\{x\in\partial\Omega:\mid x'\mid =0\},
\end{equation}
where $N\geqslant2$, $p>1$, $x'=(x_2, ... , x_N)$, $L$ is a straight line $\{x=(x_1,0,...,0)\}$;
\begin{equation}\label{1.3}
h_\omega(s)=\exp\left(-\frac{\omega(s)}s\right)\quad \forall\,s\geqslant0.
\end{equation}
Here function $\omega(\cdot)$ satisfies the following conditions:
\begin{equation}\label{1.4}
\begin{split}
\text{(i)}&\quad \omega\in C(0,\infty)\text{ is a positive nondecreasing function},
\\\text{(ii)}&\quad
s\to\mu(s):=\frac{\omega(s)}s\text{ is monotonically decreasing  on }\mathbb{R}^1_+,
\\\text{(iii)}&\quad \lim_{s\to0}\mu(s)=\infty.
\end{split}
\end{equation}

Thus $h_\omega(\mid x'\mid)$ is the absorption potential of equation \eqref{1.1} which degenerates on the line $L$ from \eqref{1.2}. If
\begin{equation}\label{1.5}
P_0(x,z)=c_Nx_N\mid x-z\mid^{-N},\ c_N=\pi^{-\frac N2}\Gamma\left(\frac N2\right),
\end{equation}
is the Poisson kernel for $-\Delta$ in $\mathbb{R}^N_+$, then (see \cite{MV1}) inequality
\begin{equation}\label{1.6}
\int_{\{\mid x\mid<R,x_N>0\}}h_\omega(\mid x'\mid)P_0(x,a)^px_Ndx<\infty\quad
\forall\,R:0<R<\infty
\end{equation}
guarantees the existence of a unique solution of the problem \eqref{1.1}, \eqref{1.2} dominated by the supersolution $KP_0(x,a)$. Thus, if condition \eqref{1.6} holds, then for an arbitrary monotonically increasing sequence
\begin{equation}\label{1.7}
\left\{K_j\right\},\ K_j\to\infty\text{ as }j\to\infty
\end{equation}
there exists a monotonically nondecreasing (due to the comparison principle) sequence of solutions $u_j(x)$ of the problem \eqref{1.1}, \eqref{1.2} with $K=K_j$. Moreover, since $h_\omega(\mid x'\mid )$ is a positive function in $\overline{\Omega}\setminus L$, equation \eqref{1.1} possesses a maximal solution $U$ in $\Omega$, which is a large solution (see \cite{RRV}):
\begin{equation}\label{1.8}
\lim_{x_N\to0,\mid x\mid <M}U(x)=\infty\quad\forall\,M>0.
\end{equation}
Since $u_j(x)\leqslant U$ $\forall\,x\in\Omega$ $\forall\,j\in\mathbb{N}$, the mentioned sequence converges to some function $u_\infty$, which is a positive solution of \eqref{1.1}.

\begin{theorem}\label{Th.1}
Let the parameter $p$ in equation \eqref{1.1} additionally satisfy
\begin{equation}\label{1.9*}
1<p<p_0:=1+\frac2{N-1}
\end{equation}
and $\{u_j(x)\}$ be a sequence of solutions of problem \eqref{1.1}, \eqref{1.2}, corresponding to $K=K_j$ from \eqref{1.7}. Assume that functions $\omega(s)$, $\mu(s)$ satisfy conditions \eqref{1.4} and
\begin{equation}\label{1.9}
\limsup_{j\to\infty}\mu\left(2^{-j+1}\right)\mu\left(2^{-j}\right)^{-1}<1;\quad\omega(s)\to0\text{ as }s\to0.
\end{equation}
Assume also that Dini condition \eqref{1.24} is not satisfied, namely:
\begin{equation}\label{1.10}
\int_0^1s^{-1}\omega(s)ds=\infty.
\end{equation}
Then $u_\infty(x):=\lim_{j\to\infty}u_j(x)$ is a solution of \eqref{1.1} which satisfies
\begin{equation}\label{1.11}
u_\infty\mid_{\partial\Omega\setminus L}=0,\quad u_\infty\mid_{\partial\Omega\cap L}=\infty.
\end{equation}
\end{theorem}

\begin{remark}\label{Rem2**}
It is clear that the problem:
\begin{equation*}
-\Delta u + h_\omega(dist(x,L))u^p=0\quad\text{in }\Omega=\mathbb{R}^N_+,
\end{equation*}
\begin{equation*}
u\mid _{\partial\Omega=\{x_N=0\}}=K\delta_a(x),\quad a\in L,
\end{equation*}
where $L$ is an arbitrary straight line in $\partial\Omega=\mathbb{R}^{N-1}$, can be transformed into a problem of the form \eqref{1.1}, \eqref{1.2}, using corresponding linear orthonormal change of variables $(x_1,...,x_N)$. Therefore, the conclusion of Theorem~\ref{Th.1} is true for solutions $u_j(x)$ of the mentioned problem with $K=K_j$ too.
\end{remark}

\begin{remark}\label{Rem2*}
Let us consider additionally the following problem:
\begin{equation}\label{1.1*}
-\Delta u + h_\omega(x_N)u^p=0\quad\text{in }\mathbb{R}^N_+,
\end{equation}
\begin{equation}\label{1.2*}
u\mid_{x_N=0}=K_j\delta_a(x),\ a\in\mathbb{R}^{N-1},\ K_j\to\infty\text{ as }j\to\infty,
\end{equation}
where $h_\omega(s)$ is the same as in \eqref{1.3}. Since $h_\omega(x_N)\leqslant h_\omega(dist(x,L))$ and $h_\omega(x_N)$ degenerates on the whole hyperplane $\{x:x_N=0\}$, then due to Theorem~\ref{Th.1} and the comparison principle, solution $u_\infty(x):=\lim_{j\to\infty}u_j(x)$ satisfies: $u_\infty(x'',0)=\infty$ $\forall\,x''\in\mathbb{R}^{N-1}$. Moreover, in \cite{ShV1} it was proved that if $\omega(s)$ from \eqref{1.3} satisfies condition \eqref{1.24} instead of \eqref{1.10}, then
\begin{equation}\label{1.12}
u_\infty(x'',0)=0\quad\forall\,x''\in\mathbb{R}^{N-1}: \ x''\neq a.
\end{equation}
Thus, the Dini condition \eqref{1.24} is a necessary and sufficient condition for the existence of the very singular solution $u_\infty(x)$ with point singularity.
\end{remark}

\begin{remark}\label{Rem1}
Condition $\omega(s)\to0$ as $s\to0$ is technical for our proof of Theorem~\ref{Th.1} and can be omitted by simple arguments. Let $\omega(s)\geqslant\omega_0=const>0$ $\forall\,s>0$. Then we can find a continuous nondecreasing function $\tilde{\omega}(s)\geqslant0$:
\begin{equation*}
\tilde{\omega}(s)>0\ \forall\,s>0,\ \tilde{\omega}(s)\to0\text{ as }s\to0, \tilde{\omega}(s)\leqslant\omega_0\ \forall\,s>0,
\end{equation*}
which satisfies condition \eqref{1.10}. Let now $\tilde{u}_j(x)$ be a sequence of solutions to problem \eqref{1.1*}, \eqref{1.2*} with absorption potential $h_{\tilde{\omega}}(x_N):=\exp\left(-\frac{\tilde{\omega}(x_N)}{x_N}\right)$ instead of $h_{\omega}(x_N)$. Then due to Th.\ref{Th.1} \ $\tilde{u}_\infty(x'',0)=\infty$ for an arbitrary $x''\in\mathbb{R}^{N-1}$. If now $u_j^{(0)}(x)$ be a sequence of solutions of problem \eqref{1.1*}, \eqref{1.2*} with $h_{\omega_0}(x_N):=\exp\left(-\frac{\omega_0}{x_N}\right)$ instead of $h_{\omega}(x_N)$, then by comparison principle
\begin{equation*}
u_j(x)\geqslant u^{(0)}_j(x)\geqslant\tilde{u}_j(x)\quad\forall\,j\in\mathbb{N},\ \forall\,x\in\overline{\Omega}.
\end{equation*}
Therefore $\infty=\tilde{u}_\infty(x'',0)\leqslant u_\infty^{(0)}(x'',0)\leqslant u_\infty(x'',0)$ and, as consequence, $u_\infty^{(0)}(x'',0)=u_\infty(x'',0)=\infty$. Notice that this last property of propagation of the strong point singularity of solution $u_\infty^{(0)}(x)$ along the whole boundary of the domain, when $\omega=\omega_0>0$, was firstly discovered by M.~Marcus, L.~Veron \cite{MV2}.
\end{remark}

The paper is organized as follows. Section 2 is devoted to the proof of the main auxiliary Theorem~\ref{Th.2}, where our variant of the local energy estimate method is applied for the study of the asymptotic behavior of solutions to semilinear elliptic equations of diffusion-absorption type near the singularity set. In section~3 the technique, elaborated in section~2, is adapted to the proof of Theorem~\ref{Th.3}. Finally, in section~4 Theorem~\ref{Th.1} about the necessity of the Dini condition is proved.

\section{Local energy estimates near the boundary singularity set}

Let $\Omega\subset\mathbb{R}^{N}_+$ be a bounded domain with $C^2$--boundary $\partial\Omega$, such that
\begin{equation}\label{1.14}
\Gamma_{\overline{R}+\rho_0}:=\{(x'',0):\mid x''\mid\leqslant\overline{R}+\rho_0\}\subset\partial\Omega,\quad \Gamma_{\overline{R}+\rho_0}\times(0,\rho_0)\subset\Omega,
\end{equation}
where $\overline{R}>0$, $\rho_0>0$.
Let $G_i$, $i=1,2,...,l$, be bounded subdomains of hyperplane $\{x_N=0\}$ with $C^2$--boundaries $\partial G_i$, such that
\begin{equation}\label{1.15}
G_i\subset\{\mid x''\mid <\overline{R}\}\quad\forall\,i\leqslant l,
\end{equation}
\begin{equation}\label{1.16}
dist(G_i,G_j):=\inf_{x\in G_i,y\in G_j}\mid x-y\mid >\rho_0\quad\forall\,i\neq j,\ \rho_0\text{ is from \eqref{1.14}}.
\end{equation}
In this domain $\Omega$ we consider the following boundary Dirichlet problem:
\begin{equation}\label{1.18}
u\mid _{\bar{G}_i}=K^{(i)}=const>0,\
i=1,2,...,l; \quad
u=0\text{ on }\partial\Omega\setminus\biggr\{\bigcup_{i=1}^l\bar{G}_i\biggr\},
\end{equation}
for equation \eqref{1.17}.
Introduce now $l$ sequences
\begin{equation}\label{1.21}
\{K_j^{(i)}\},\ i\leqslant l,\ j=1,2,...:\ K_j^{(i)}\to\infty\text{ as }j\to\infty \quad\forall\,i\leqslant l,
\end{equation}
and let $\{u_j\}$, $j=1,2,...$, be an infinite sequence of solutions of equation \eqref{1.17} satisfying the boundary condition
\begin{equation}\label{1.22}
u_j\mid _{\bar{G}_i}=K_j^{(i)},\quad
u_j=0\text{ on }\partial\Omega\setminus\left\{\cup_{i=1}^l\bar{G}_i\right\}.
\end{equation}

\begin{theorem}\label{Th.2}
Let functions $h_\omega(\cdot)$ and $H(\cdot)$ satisfy relation \eqref{1.20} and let $\omega$ from \eqref{1.20} be a nondecreasing continuous function satisfying technical condition \eqref{1.23} and Dini condition \eqref{1.24}. If $u_j$ is a solution of  problem \eqref{1.17}, \eqref{1.22}, then $u_\infty:=\lim_{j\to\infty}u_j$ is a solution of equation \eqref{1.17}, satisfying the boundary conditions
\begin{equation}\label{1.25}
\lim_{x\to y}u_\infty(x)=0\quad\forall\,y\in \partial\Omega\setminus
\big\{\cup_{i\leqslant l}\overline{G}_i\big\},\quad \lim_{x\to y}u_\infty(x)=\infty\quad\forall\,y\in
\cup_{i\leqslant l}\overline{G}_i
\end{equation}
\end{theorem}

%The proof of Theorem~\ref{Th.2} has in some sense universal character. Particularly, it can be adapted for the obtaining of following result about uniqueness of large solution to equation \eqref{1.17}.

%\begin{theorem}\label{Th.3}
%Let us consider equation \eqref{1.17} in arbitrary bounded domain $\Omega\subset\mathbb{R}^N$ with $C^2$ boundary $\partial\Omega$ and let potential $H(x)$ satisfies estimate \eqref{1.20}, where function $\omega(s)$ satisfies condition \eqref{1.23} and Dini condition \eqref{1.24}. Then equation \eqref{1.17} admit only one large solution in the domain under consideration.
%\end{theorem}

\begin{proof}

Let us introduce the following families of subdomains of $\Omega$ from \eqref{1.14}--\eqref{1.16}:
\begin{equation}\label{3.1}
\begin{split}
  & \Omega_s:= \{x\in\Omega:d(x)>s\}\quad\forall\,s\in\mathbb{R}_+^1, \\
&   \Omega^s:= \{x\in\Omega:0<d(x)<s\}\quad\forall\,s\in\mathbb{R}_+^1
\end{split}
\end{equation}
Due to the smoothness of $\partial\Omega$ there exists $\bar{s}>0$, such that $\partial\Omega^s\cap\Omega=\partial\Omega_s$ is $C^2$--smooth for any $s:0<s<\bar{s}$. Moreover, we can assume that
\begin{equation}\label{3.2}
d(x)=x_N\quad\forall\,x\in\Gamma_{\bar{R}+\rho_0}\times(0,\rho_0).
\end{equation}
Let $u$ be a nonnegative solution of equation \eqref{1.17} in $\Omega$. Introduce the following energy function, connected with $u$:
\begin{equation}\label{3.3}
I(s):=\int_{\Omega_s}\left(\mid \nabla_xu\mid ^2+h_\omega (d(x))u^{p+1}\right)dx,\quad s>0.
\end{equation}

\begin{lemma}\label{Lem3.1}
The function $I(\cdot)$ from \eqref{3.3} satisfies the estimate:
\begin{equation}\label{3.4}
I(s)\leqslant d_3\left[\int_0^sh_\omega(r)^{\frac2{p+3}}dr\right]^{-\frac{p+3}{p-1}}, \quad\forall\,s:0<s<\bar{s},
\end{equation}
where constant $d_3<\infty$ does not depend on $u$.
\end{lemma}
\begin{proof}
Multiplying equation \eqref{1.17} by $u$ and integrating it over $\Omega_s$, $s:0<s<\bar{s}$, we obtain:
\begin{multline}\label{3.5}
\int_{\Omega_s}\left(\sum_{i,j=1}^Na_{ij}(x)u_{x_i}u_{x_j} +H(x)u^{p+1}\right)dx = \int_{\partial\Omega_s} \sum_{i,j=1}^Na_{ij}(x)u_{x_i}u\nu_jd\sigma\leqslant\\
\leqslant\left(\int_{\partial\Omega_s} \sum_{i,j=1}^Na_{ij}(x)u_{x_i}u_{x_j}d\sigma\right)^\frac12 \left(\int_{\partial\Omega_s} \sum_{i,j=1}^Na_{ij}(x)\nu_i\nu_ju^2d\sigma\right)^\frac12,
\end{multline}
where $\nu(x)=(\nu_1,...,\nu_N)$ is an outward normal unit vector to $\partial\Omega$.
By \eqref{1.19}, \eqref{1.20} and H\"{o}lder's inequality, we have:
\begin{equation*}
\left(\int_{\partial\Omega_s} \sum_{i,j=1}^Na_{ij}(x)\nu_i\nu_ju^2d\sigma\right)^\frac12\leqslant c(meas\,\partial\Omega_s)^\frac{q-1}{2(q+1)}h_\omega(s)^{-\frac{1}{q+1}} \left(\int_{\partial\Omega_s} h_\omega(s)u^{q+1}d\sigma\right)^\frac1{q+1}.
\end{equation*}
Substituting this estimate into \eqref{3.5} and using Young's inequality we obtain:
\begin{equation}\label{3.6}
I(s)\leqslant c_1h_\omega(s)^{-\frac{1}{p+1}}\left(\int_{\partial\Omega_s}\left(\mid \nabla_xu\mid ^2+h_\omega \left(d(x)\right)u^{p+1}\right)d\sigma\right)^{1-\frac{p-1}{2(p+1)}}.
\end{equation}
It is easy to see that
\begin{equation*}
\frac{dI(s)}{ds}= -\int_{\partial\Omega_s}\left(\mid \nabla_xu\mid ^2+h_\omega \left(d(x)\right)u^{p+1}\right)d\sigma.
\end{equation*}
Substituting this relation into \eqref{3.6} we derive the following differential inequality:
\begin{equation*}
I(s)\leqslant c_2h_\omega(s)^{-\frac{1}{p+1}}\left(-I'(s)\right)^{1-\frac{p-1}{2(p+1)}}.
\end{equation*}
Solving this inequality we obtain \eqref{3.4}.
\end{proof}

Now we derive the global upper a priori estimates for solutions $u_j$ of the problem \eqref{1.17}, \eqref{1.22} when $j\to\infty$. For an arbitrary small $\delta>0$ we introduce $C^1$--smooth function $\xi_\delta(x'')$ with $supp\ \xi_\delta\subset\{x''\in\mathbb{R}^{N-1}:\mid x''\mid \leqslant\overline{R}+\delta\}$, $\delta<2^{-1}\rho_0$, such that:
\begin{align}
&\xi_\delta(x'')=1\text{ if }x''\in\overline{G}_i\quad\forall\,i\leqslant l,\label{3.7}
\\&\xi_\delta(x'')=0\text{ if }dist(x'',\overline{G}_i):=\min_{y\in\overline{G}_i}\mid x''-y\mid \geqslant\delta \quad\forall\,i\leqslant l,\label{3.8}
\\&0\leqslant\xi_\delta(x'')\leqslant1\quad\forall\,x''\in\mathbb{R}^{N-1}\setminus \cup_{i=1}^l\overline{G}_i:\min_{i\leqslant l} dist(x'',\overline{G}_i)<\delta. \label{3.9}
\end{align}
It is clear that
\begin{equation}\label{3.10}
\mid \nabla\xi_\delta\mid \leqslant c\delta^{-1}\quad\forall\,\delta:0<\delta<2^{-1}\rho_0,
\end{equation}
where $c<\infty$ does not depend on $\delta$.
Let $u_{j,\delta}$, $j=1,2,...$, be a solution of equation \eqref{1.17} satisfying the regularized boundary condition:
\begin{equation}\label{3.11}
u_{j,\delta}=K_j\xi_\delta\text{ on }\partial\Omega,\quad K_j:=\max_{i\leqslant l}K_j^{(i)}.
\end{equation}
By the comparison principle we have:
\begin{equation}\label{3.12}
u_{j,\delta}(x)\geqslant u_j(x)\quad\forall\,x\in\overline{\Omega},\ \forall\,j\in\mathbb{N},\ \forall\,\delta:0<\delta<2^{-1}\rho_0.
\end{equation}
Therefore, to prove Theorem~\ref{Th.2} it is sufficient to investigate and estimate from above the solution $u_{j,\delta}$ with an arbitrary small $\delta>0$. For the sake of simplicity of the notations we omit $\delta$ in $u_{j,\delta}$ and\ denote $u_{j,\delta}$ by $u_j$.

\begin{lemma}\label{Lem3.2}
Solution $u_j=u_{j,\delta}$ of problem \eqref{1.17}, \eqref{3.11} satisfies the following estimate:
\begin{equation}\label{3.13}
\int_{\Omega}\left(\mid \nabla u_j\mid ^2+h_\omega \left(d(x)\right)u_j^{p+1}\right)dx\leqslant \overline{K}_j:=\bar{c}\big(K_j^{p+1}+\delta^{-1}K_j^2\big),
\end{equation}
where constant $\bar{c}<\infty$ does not depend on $j\in\mathbb{N}$, $\delta\in(0,2^{-1}\rho_0)$.
\end{lemma}
\begin{proof}
Let us introduce $C^2$--cut--off function $\zeta=\zeta_\delta(s)$, such that $\zeta_\delta(s)=1$ if $s\leqslant\delta$, $\zeta_\delta(s)=0$ if $s>2\delta$, $0\leqslant\zeta_\delta(s)\leqslant1$, $\mid \nabla\zeta_\delta\mid <c\delta^{-1}$. Multiplying \eqref{1.17} by
\begin{equation}\label{3.13*}
v_j(x)=u_j(x)-K_j\xi_\delta(x'')\zeta_\delta(d(x)),\quad K_j\text{ is from \eqref{3.11}},
\end{equation}
and integrating it over $\Omega$, due to $v_j=0$ on $\partial\Omega$ we obtain:
\begin{equation}\label{3.14}
\begin{split}
&\int_{\Omega}\left(\sum_{i,k=1}^Na_{ik}(x)u_{jx_i}u_{jx_k} +H(x)u_j^{p+1}\right)dx = \\=&\int \sum_{i,k=1}^Na_{ik}(x)u_{jx_i}\left(\xi_\delta(x'')\zeta_\delta(d(x))\right)_{x_k}K_jdx+\\
+&\int_{\Omega}K_jH(x)u_j^p\xi_\delta(x'')\zeta_\delta(d(x))dx:= A_1+A_2.
\end{split}
\end{equation}
By Young's inequality and properties \eqref{3.7}--\eqref{3.10} we get:
\begin{equation}\label{3.15}
\begin{split}
  & \mid A_1\mid \leqslant2^{-1} \int_{\Omega}\sum_{i,k=1}^Na_{ik}(x)u_{jx_i}u_{jx_k}dx+c\delta^{-1}K_j^2, \\
&   \mid A_2\mid \leqslant2^{-1}\int_{\Omega}H(x)u_j^{p+1}dx+ c'K_j^{p+1},
\end{split}
\end{equation}
where constants $c,c'<\infty$ do not depend on $\delta$, $j$. By \eqref{3.15} and \eqref{3.14} we have:
\begin{equation}\label{3.16}
\int_{\Omega}\left(\sum_{i,k=1}^Na_{ik}(x)u_{jx_i}u_{jx_k} +H(x)u_j^{p+1}\right)dx\leqslant\bar{\bar{c}} (K_j^{p+1}+\delta^{-1}K_j^2),\quad \bar{\bar{c}}=\max(c,c'),
\end{equation}
which yields the estimate \eqref{3.13} due to properties \eqref{1.20}, \eqref{1.19}.
\end{proof}

Introduce now the following family of subdomains of the domain $\Omega^s$ with an arbitrary $s\in(0,\rho_0)$:
\begin{equation}\label{3.14*}
\Omega^s(\tau):=\Omega^s\setminus\big\{x=(x'',x^N)\in\Omega^s:r(x''):= \min_{i\leqslant l}dist(x'',G_i)<\tau\big\} \quad\forall\,\tau\in\big(0,\frac{\rho_0}{2}\big),
\end{equation}
where $\rho_0>0$ is from \eqref{3.2}. Introduce also another family of energy functions for the solution $u_j=u_{j,\delta}$ under consideration:
\begin{equation}\label{3.15*}
J_j(s,\tau):=\int_{\Omega^{2s}(\tau)}\left(\mid \nabla u_j\mid ^2+h_\omega(d(x))u_j^{p+1}\right)\zeta_s(d(x))dx,
\end{equation}
where $\zeta_s(\cdot)$ is a function from \eqref{3.13*}: $\zeta_s(d)=1$ if $d\leqslant s$, $\zeta_s(d)=0$ if $d>2s$, $0\leqslant\zeta_s(d)\leqslant1$, $\mid \nabla \zeta_s\mid \leqslant cs^{-1}$.

\begin{lemma}\label{Lem3.3}
The energy function $J_j(s,\tau)$ from \eqref{3.15*} satisfies the following differential inequality
\begin{multline}\label{3.16*}
J_j(s,\tau)\leqslant cs\left(-\frac{d}{d\tau}J_j(s,\tau)\right)+Ch_\omega(s)^{-\frac2{p-1}-\nu}\quad \forall\,\tau\in\big(\delta,\frac{\rho_0}{2}\big),\ \forall\,j\in\mathbb{N},
\\ \forall\,s\in\big(0,\frac{\rho_0}{2}\big),\ \forall\,\nu>0,\ C=C(\nu)\to\infty\text{ as }\nu\to0,
\end{multline}
where constants $c,C$ do not depend on $j$.
\end{lemma}
\begin{proof}
We multiply equation \eqref{1.17} for the solution $u_j(x)$ by $u_j(x)\zeta_s(d(x))$ and integrate it over $\Omega^{2s}(\tau)$, $\tau>\delta$. As a result we obtain the following relation:
\begin{equation}\label{3.17*}
\begin{split}
\overline{J}_j(s,\tau):=&\int_{\Omega^{2s}(\tau)} \left(\sum_{i,k=1}^Na_{ik}(x)u_{jx_i}u_{jx_k} +H(x)u_j^{p+1}\right)\zeta_s(d(x))dx=\\=&R_1+R_2:=\int_{\Gamma^{2s}(\tau)} \sum_{i,k=1}^Na_{ik}(x)u_{jx_i}u_{j}\nu_k(x)\zeta_s(d(x))d\sigma- \\-&\int_{\Omega^{2s}(\tau)\setminus\Omega^s(\tau)} \sum_{i,k=1}^Na_{ik}(x)u_{jx_i}\zeta_s(d(x))_{x_k}u_jdx,
\end{split}
\end{equation}
where $\Gamma^{2s}(\tau):=\bigcup_{i\leqslant l}\Gamma_i^{2s}(\tau)$, $\Gamma_i^{2s}(\tau):=\{x=(x'',x_N):x_N<2s,dist(x'',G_i)=\tau\}$. Notice that due to \eqref{3.2} and \eqref{1.15}, \eqref{1.16} we have: $\Gamma_i^{2s}(\tau)\cap\Gamma_j^{2s}(\tau)=\emptyset$ $\forall\,i\neq j$, $\forall\,\tau<2^{-1}\rho_0$. Now using H\"{o}lder's inequality we estimate $R_1$ from above:
\begin{equation}\label{3.18}
\begin{split}
\mid R_1\mid \leqslant & \left(\int_{\Gamma^{2s}(\tau)}\sum_{i,k=1}^Na_{ik}u_{jx_i}u_{jx_k}\zeta_s(d(x))d\sigma\right)^\frac12
\times \\ \times & \left (\int_{\Gamma^{2s}(\tau)}\sum_{i,k=1}^Na_{ik}\nu_i\nu_ku_j^2\zeta_s(d(x))d\sigma\right)^\frac12= (R_1^{(1)})^\frac12 (R_1^{(2)})^\frac12.
\end{split}
\end{equation}
By \eqref{1.19} we estimate $R_1^{(2)}$:
\begin{equation*}
\begin{split}
R_1^{(2)}\leqslant &d_1\int_{\Gamma^{2s}(\tau)}u_j^2\zeta_s(d(x))d\sigma= d_1\left(\int_{\Gamma^{2s}(\tau)\setminus\Gamma^s(\tau)}u_j^2\zeta_s(d(x))d\sigma+
\int_{\Gamma^{s}(\tau)}u_j^2d\sigma\right)=\\=&d_1 \left(R_{1,1}^{(2)}+R_{1,2}^{(2)}\right).
\end{split}
\end{equation*}
Since $u_j(x'',0)=0$ $\forall\,x''\in\Gamma^s(\tau):\delta<\tau<\rho_0$, we derive by the Poincare's inequality:
\begin{multline}\label{3.19}
R_{1,2}^{(2)}=\int_{\Gamma^s(\tau)}u_j^2d\sigma\leqslant d_2s^2\int_{\Gamma^s(\tau)}\mid \frac{\partial u_j}{\partial x_N}\mid ^2d\sigma\leqslant d_2s^2\int_{\Gamma^s(\tau)}\mid \nabla u_j\mid ^2d\sigma \\ \forall\,\tau:\delta<\tau<\frac{\rho_0}{2}.
\end{multline}
We estimate the term $R_{1,1}^{(2)}$ by the standard trace interpolation inequality (see e.g. \cite{GT}):
\begin{equation}\label{3.20}
\begin{split}
  &\int_{\Gamma_{i,x_N}(\tau)}u_j(x'',x^N)^2d\sigma''\leqslant c_1\left(\int_{\tau<\mid x''\mid <\rho_0}\mid \nabla_{x''} u_j(x'',x_N)\mid ^2dx''\right)^\frac12 \times \\&\times\left(\int_{\tau<\mid x''\mid <\rho}u_j(x'',x_N)^2dx''\right)^\frac12 +c_2\int_{\tau<\mid x''\mid <\rho_0}u_j(x'',x_N)^2dx'' \\ &
\forall\,\tau:\delta<\tau<\frac{\rho_0}{2},\ \forall\,x_N\in(s,2s),\ \forall\,i\leqslant l,
\end{split}
\end{equation}
where $\Gamma_{i,x_N}(\tau):=\{x=(x'',x_N):dist(x'',G_i)=\tau,x_N=const\}$, constants $c_1,c_2$ do not depend on $\tau,s$. Integrating the last inequality with respect to $x_N$ over the interval $(s,2s)$ and summing obtained inequalities from $i=1$ up to $i=l$, we obtain after simple computations:
\begin{equation}\label{3.21}
\begin{split}
R_{1,1}^{(2)}\leqslant & c_1\left(\int_{\Omega^{2s}(\tau)\setminus\Omega^s(\tau)}\mid \nabla_{x''} u_j\mid ^2dx\right)^\frac12 \left(\int_{\Omega^{2s}(\tau)\setminus\Omega^s(\tau)}u_j(x)^2dx\right)^\frac12+
\\ +&c_2\int_{\Omega^{2s}(\tau)\setminus\Omega^s(\tau)}u_j(x)^2dx=: c_1(R_{1,1,1}^{(2)})^\frac12(R_{1,1,2}^{(2)})^\frac12+ c_2R_{1,1,2}^{(2)} \\ &\forall\,\tau\in(\delta,2^{-1}\rho_0),\forall\,s\in(0,2^{-1}\rho_0).
\end{split}
\end{equation}
By H\"{o}lder's inequality we get:
\begin{equation}\label{3.22}
R_{1,1,2}^{(2)}\leqslant c_3s^{\frac{p-1}{p+1}} h_\omega(s)^{-\frac{2}{p+1}}\left(\int_{\Omega^{2s}(\tau)\setminus\Omega^s(\tau)} h_\omega(d(x))u_j(x)^{p+1}dx\right)^{\frac{2}{p+1}}.
\end{equation}
It follows from \eqref{3.21}, \eqref{3.22} that
\begin{equation}\label{3.23}
\begin{split}
  R_{1,1}^{(2)}\leqslant &c_4s^{\frac{p-1}{p+1}} h_\omega(s)^{-\frac{2}{p+1}}\left(\int_{\Omega^{2s}(\tau)\setminus\Omega^s(\tau)} h_\omega(d(x))u_j(x)^{p+1}dx\right)^{\frac{2}{p+1}}+\\
  &+c_5s^{\frac{p-1}{2(p+1)}} h_\omega(s)^{-\frac{1}{p+1}}\left(\int_{\Omega^{2s}(\tau)\setminus\Omega^s(\tau)} \mid \nabla_xu_j\mid ^2dx\right)^{\frac12}\times \\ \times& \left(\int_{\Omega^{2s}(\tau)} h_\omega(d(x))u_j(x)^{p+1}dx\right)^{\frac{1}{p+1}}
  \leqslant \\ \leqslant&c_4s^{\frac{p-1}{p+1}} h_\omega(s)^{-\frac{2}{p+1}}\left(I_j(s)-I_j(2s)\right)^{1-\frac{p-1}{p+1}}+ \\ +&
  c_5s^{\frac{p-1}{2(p+1)}} h_\omega(s)^{-\frac{1}{p+1}} \left(I_j(s)-I_j(2s)\right)^{1-\frac{p-1}{2(p+1)}},
\end{split}
\end{equation}
where $I_j(s)=\int_{\Omega_s} \left(\mid \nabla_xu_j\mid ^2+ h_\omega(d(x))u_j(x)^{p+1}\right)dx$.
Plugging estimates \eqref{3.19} and \eqref{3.23} into \eqref{3.18} and using Young's inequality we obtain:
\begin{equation}\label{3.24}
\begin{split}
  \mid R_1\mid \leqslant &c_6\left(\int_{\Gamma^{2s}(\tau)} \mid \nabla u_j\mid ^2\zeta_sd\sigma\right)^{\frac12} \biggr[s^2\int_{\Gamma^{s}(\tau)}\mid \nabla u_j\mid ^2d\sigma + \\ +& s^{\frac{p-1}{p+1}}h_\omega(s)^{-\frac{2}{p+1}}\left(I_j(s)-I_j(2s)\right)^{1-\frac{p-1}{p+1}}+
  \\+&
  s^{\frac{p-1}{2(p+1)}} h_\omega(s)^{-\frac{1}{p+1}} \left(I_j(s)-I_j(2s)\right)^{1-\frac{p-1}{2(p+1)}}\biggr]^\frac12\leqslant
  \\ \leqslant & c_7\biggr[s\int_{\Gamma^{2s}(\tau)}\mid \nabla u_j\mid ^2\zeta_sd\sigma+ \\+ & s^{-1+\frac{p-1}{2(p+1)}}h_\omega(s)^{-\frac{1}{p+1}}\left(I_j(s)-I_j(2s)\right)^{1-\frac{p-1}{2(p+1)}}+
 \\+ &s^{-1+\frac{p-1}{p+1}} h_\omega(s)^{-\frac{2}{p+1}} \left(I_j(s)-I_j(2s)\right)^{1-\frac{p-1}{p+1}}\biggr].
\end{split}
\end{equation}
Finally, we estimate $R_2$. Using H\"{o}lder's inequality and property \eqref{1.19} we get:
\begin{equation}\label{3.25}
\begin{split}
\mid R_2\mid \leqslant & cs^{-1}\left(\int_{\Omega^{2s}(\tau)\setminus\Omega^{s}(\tau)} \mid \nabla u_j\mid ^2dx\right)^{\frac12} \left(\int_{\Omega^{2s}(\tau)\setminus\Omega^{s}(\tau)} u_j^2dx\right):=\\:=&cs^{-1}(R_2^{(1)})^\frac12(R_2^{(2)})^\frac12.
\end{split}
\end{equation}
The term $R_2^{(2)}$ coincides with $R_{1,1,2}^{(2)}$ and it can be estimated as in \eqref{3.22}. Therefore, by Young's inequality we get from \eqref{3.25} that
\begin{equation}\label{3.26}
\mid R_2\mid \leqslant cs^{-\left(1-\frac{p-1}{2(p+1)}\right)}h_\omega(s)^{-\frac1{p+1}}
\left(\int_{\Omega^{2s}(\tau)\setminus\Omega^{s}(\tau)} \left(\mid \nabla u_j\mid ^2+h_\omega(d(x))u_j^{p+1}\right)dx\right)^{1-\frac{p-1}{2(p+1)}}.
\end{equation}
Thus, due to estimates \eqref{3.24} and \eqref{3.26} it follows from \eqref{3.17*} that
\begin{multline}\label{3.27}
J_j(s,\tau)\leqslant d_0^{-1}\overline{J}_j(s,\tau)\leqslant cs\int_{\Gamma^{2s}(\tau)}\mid \nabla_xu_j\mid ^2\zeta_s(d(x))d\sigma+ c_1s^{-\frac2{p+1}}h_\omega(s)^{-\frac2{p+1}}\times\\ \times\left(I_j(s)-I_j(2s)\right)^{1-\frac{p-1}{p+1}}+
c_2s^{-\frac{p+3}{2(p+1)}}h_\omega(s)^{-\frac1{p+1}}\left(I_j(s)-I_j(2s)\right)^{1-\frac{p-1}{2(p+1)}}.
\end{multline}
It is easy to check that
\begin{equation}\label{3.28}
\int_{\Gamma^{2s}(\tau)}\left(\mid \nabla u_j\mid ^2+h_\omega(d(x)) u_j^{p+1}\right)\zeta_s(d(x)) d\sigma\leqslant-\bar{c}\frac{d}{d\tau}J_j(s,\tau),
\end{equation}
where $\bar{c}=const>0$ does not depend on $\tau$, $s$, $j$. Substituting \eqref{3.28} into \eqref{3.27} we obtain:
\begin{equation}\label{3.29}
\begin{split}
&J_j(s,\tau)\leqslant \bar{\bar{c}}s\left(-\frac{d}{d\tau}J_j(s,\tau)\right)+c_1F_j(s)\quad\forall\, \tau\in\left(\delta,\frac{\rho_0}{2}\right),\ \forall\,s\in\left(0,\frac{\rho_0}{2}\right),
\\& F_j(s):=\frac{\left(I_j(s)-I_j(2s)\right)^{1-\frac{p-1}{p+1}}}{s^{\frac2{p+1}}h_\omega(s)^{\frac2{p+1}}} +\frac{\left(I_j(s)-I_j(2s)\right)^{1-\frac{p-1}{2(p+1)}}}{s^{\frac{p+3}{2(p+1)}}h_\omega(s)^{\frac1{p+1}}}
\end{split}
\end{equation}
It  only remains to estimate $F_j(s)$ from above. By lemma~\ref{Lem3.1} we have the following uniform, with respect to $j\in\mathbb{N}$, upper estimate for the energy functions $I_j$:
\begin{equation}\label{3.4*}
I_j(s)\leqslant
d_3\left[\int_0^sh_\omega(r)^{\frac2{p+3}}dr\right]^{-\frac{p+3}{p-1}}
\quad\forall\,s:0<s<\bar{s}.
\end{equation}
Since $\omega(\cdot)$ is a nondecreasing function it is easy to check (see lemma~2.4 from \cite{ShV1}) that
\begin{equation}\label{3.30}
\int_0^s\exp\left(-\frac{b\,\omega(t)}{t}\right)dt\geqslant \frac{s^2}{2s+b\,\omega(s)}\exp\left(-\frac{b\,\omega(s)}{s}\right)\quad \forall\,b>0.
\end{equation}
Therefore, by \eqref{3.30} it follows from \eqref{3.4*} that
\begin{equation}\label{3.31}
I_j(s)\leqslant
d_3\left(\frac{2s+\frac{2}{p+3}\omega(s)}{s^2}\right)^{\frac{p+3}{p-1}} \exp\left(\frac{2}{p-1}\frac{\omega(s)}{s}\right):=d_3\Phi_1(s)h_\omega(s)^{-\frac2{p-1}}.
\end{equation}
Substituting estimate \eqref{3.31} into the definition of the function $F_j(s)$ we obtain:
\begin{equation}\label{3.32}
F_j(s)\leqslant d_4\left(\left(\frac{\Phi_1(s)}{s}\right)^\frac2{p+1}+ \left(\frac{\Phi_1(s)}{s}\right)^\frac{p+3}{2(p+1)}\right) h_\omega(s)^{-\frac2{p-1}}.
\end{equation}
By condition \eqref{1.23} for the function $\omega(s)$ we have:
\begin{equation*}
h_\omega(s)^{\nu}\leqslant\exp\left(-\nu s^{-(1-\gamma_1)}\right) \qquad\forall\,s\in(0,\rho_0),\ \forall\,\nu>0,
\end{equation*}
which yields the upper estimate for $F_j(s)$:
\begin{equation}\label{3.33}
F_j(s)\leqslant
C(\nu)h_\omega(s)^{-\frac2{p-1}-\nu}\quad\forall\,\nu>0,\ C(\nu)\to\infty\text{ as }\nu\to0.
\end{equation}
\end{proof}

Now we get down to the main step of the proof of the theorem, which consists of a careful analysis of the vanishing properties of the energy functions $J_j(s,\tau)$, satisfying inequalities \eqref{3.16*} for all $j\in\mathbb{N}$. Notice that due to global estimate \eqref{3.13} function $J_j(s,\tau)$ satisfies the following "initial" condition:
\begin{equation}\label{3.34}
J_j(s,\delta)\leqslant
\overline{K}_j:=\bar{c}(K_j^{p+1}+\delta^{-1}K_j^2)\quad\forall\,j\in\mathbb{N},
\end{equation}
where $\delta>0$ and $K_j$ are from boundary condition \eqref{3.11}. Let us fix $j$ large enough and $\nu>0$ small enough. Next we define $s_j>0$ by the following relation:
\begin{equation}\label{3.35}
C(\nu)h_\omega(s_j)^{-\frac2{p-1}-\nu}=\overline{K}_j^\theta,\quad C(\nu)\text{ from \eqref{3.33}},
\end{equation}
where $0<\theta<1$ will be defined later. It follows from \eqref{3.16*}, \eqref{3.34} that $J_j(s_j,\tau)$ satisfies the following differential inequality:
\begin{equation}\label{3.36}
\begin{split}
&J_j(s_j,\tau)\leqslant \tilde{c}s_j\left(-\frac{d}{d\tau}J_j(s_j,\tau)\right)+\overline{K}_j^\theta\quad \forall\,\tau:\delta<\tau<2^{-1}\rho_0,\\
&J_j(s_j,\delta)\leqslant\overline{K}_j.
\end{split}
\end{equation}
Let us define now value $\tau_j$ by the equality:
\begin{equation}\label{3.37}
J_j(s_j,\delta+\tau_j)=2\overline{K}_j^\theta,
\end{equation}
where $\theta$ is from \eqref{3.35}. To find an upper estimate for $\tau_j$, we notice that
\begin{equation*}
J_j(s_j,\tau)>2\overline{K}_j^\theta\quad\forall\,\tau\in(\delta,\delta+\tau_j).
\end{equation*}
Hence \eqref{3.36} yields:
\begin{equation}\label{3.38}
J_j(s_j,\tau)\leqslant2\tilde{c}s_j\left(-\frac{d}{d\tau}J_j(s_j,\tau)\right) \quad\forall\,\tau\in(\delta,\delta+\tau_j).
\end{equation}
Solving this differential inequality and taking into account the initial condition in \eqref{3.36}, we obtain:
\begin{equation}\label{3.39}
J_j(s_j,\tau)\leqslant\overline{K}_j\exp\left(-\frac{\tau-\delta}{2\tilde{c}s_j}\right) \quad\forall\,\tau\in(\delta,\delta+\tau_j).
\end{equation}
By \eqref{3.37} and \eqref{3.39} we get: $2\overline{K}_j^\theta\leqslant\overline{K}_j \exp\left(-\frac{\tau_j}{2\tilde{c}s_j}\right)$, where $\tilde{c}$ is a constant from \eqref{3.36}. Hence, $\tau_j$ satisfies:
\begin{equation}\label{3.40}
0<\tau_j\leqslant2\tilde{c}s_j(-\ln2+(1-\theta)\ln\overline{K}_j).
\end{equation}
Next, notice that by definitions \eqref{3.3}, \eqref{3.15*} we have:
\begin{equation}\label{3.41}
\int_{\Omega^{\rho_0}(\delta+\tau_j)}\left(\mid \nabla u_j\mid ^2+h_\omega(d(x))u_j^{p+1}\right)dx \leqslant I_j(s_j)+J_j(s_j,\delta+\tau_j),\text{ if }\delta+\tau_j<2^{-1}\rho_0.
\end{equation}
Due to estimate \eqref{3.31} and condition \eqref{1.23} on $\omega(s)$, analogously to \eqref{3.33}, we have:
\begin{equation}\label{3.42}
I_j(s_j)\leqslant C_1(\nu)h_\omega(s_j)^{-\frac2{p-1}-\nu}, \quad\forall\,\nu>0,\ C_1(\nu)\to\infty\text{ as }\nu\to0.
\end{equation}
Using now definition \eqref{3.35} of $s_j$ and \eqref{3.37} of $\tau_j$, we deduce from \eqref{3.41} and \eqref{3.42}:
\begin{equation}\label{3.43}
\int_{\Omega^{\rho_0}(\delta+\tau_j)}\left(\mid \nabla u_j\mid ^2+h_\omega(d(x))u_j^{p+1}\right)dx\leqslant \left(2+\frac{C_1(\nu)}{C(\nu)}\right)\overline{K}_j^\theta.
\end{equation}
Now we define sequences $\{{K}_i\}$ and $\{\overline{K}_i\}$, $i=1,2,...$, which are connected by the relation \eqref{3.34}. Firstly introduce
\begin{equation}\label{3.44}
\overline{K}_i:=\exp\exp i,\quad i\in\mathbb{N}.
\end{equation}
Then define $\{{K}_i\}=\{{K}_i(\delta,\bar{c})\}$ as solutions of algebraic equation \eqref{3.34}. It is easy to see that $K_i={K}_i(\delta,\bar{c})\to\infty$ as $i\to\infty$.
Now we have to fix parameter $\theta$ from definition \eqref{3.35} of $s_j$, namely, we have to guarantee the validity of the following inequality:
\begin{equation}\label{3.45}
\left(2+\frac{C_1(\nu)}{C(\nu)}\right)\overline{K}_j^\theta\leqslant \overline{K}_{j-1}
\end{equation}
Due to \eqref{3.44} inequality \eqref{3.45} is equivalent to:
\begin{equation}\label{3.46}
\ln\left(2+\frac{C_1(\nu)}{C(\nu)}\right)+\theta\exp j \leqslant\exp(j-1)=e^{-1}\exp j.
\end{equation}
It is easy to see that \eqref{3.46} is satisfied by
\begin{equation}\label{3.46*}
\theta=(2e)^{-1}\quad\text{if }\,j\geqslant j_0:=1+\ln2+ \ln\ln\left(2+\frac{C_1(\nu)}{C(\nu)}\right).
\end{equation}
With such $\theta$ inequality \eqref{3.43} yields:
\begin{equation}\label{3.47}
\int_{\Omega^{\rho_0}(\delta+\tau_j)}\left(\mid \nabla u_j\mid ^2+h_\omega(d(x))u_j^{p+1}\right)dx\leqslant \overline{K}_{j-1}.
\end{equation}
Now we obtain explicit upper estimates of $\tau_j$, $s_j$, defined by \eqref{3.35}, \eqref{3.37}. Firstly, \eqref{3.35} yields:
\begin{equation}\label{3.48}
\begin{split}
&C(\nu)\exp\left(\big(\frac2{p-1}+\nu\big)\frac{\omega(s_j)}{s_j}\right) =\overline{K}_j^\theta\ \Rightarrow\ \frac\theta2\ln\overline{K}_j\leqslant \left(\frac2{p-1}+\nu\right)\frac{\omega(s_j)}{s_j}\leqslant\theta\ln \overline{K_j}
\\ &\forall\,j\geqslant j'=j'(\nu)=\ln\ln C(\nu)+\ln\theta^{-1}+\ln2.
\end{split}
\end{equation}
By \eqref{3.48}, \eqref{1.23} and \eqref{3.44} we have:
\begin{equation}\label{3.49}
s_j\leqslant 2\left(\frac2{p-1}+\nu\right)\theta^{-1}(\ln \overline{K}_j)^{-1}\omega(s_j)\leqslant 2\left(\frac2{p-1}+\nu\right)\theta^{-1}\omega_0\exp(-j).
\end{equation}
This estimate due to the monotonicity of $\omega(\cdot)$ yields:
\begin{equation}\label{3.50}
\omega(s_j)\leqslant\omega(C_3\exp(-j)),\quad C_3=2\left(\frac2{p-1}+\nu\right)\theta^{-1}\omega_0.
\end{equation}
As to $\tau_j$, we get from \eqref{3.40} and \eqref{3.49} that
\begin{equation}\label{3.51}
\tau_j\leqslant 2\tilde{c}s_j(1-\theta)\ln\overline{K}_j\leqslant
4\tilde{c}(1-\theta)
\left(\frac2{p-1}+\nu\right)\theta^{-1}\omega(s_j)\leqslant C_2\omega(s_j),
\end{equation}
where $C_2=4\theta^{-1}(1-\theta)\tilde{c}
\left(\frac2{p-1}+\nu\right)$. Substituting \eqref{3.50} into \eqref{3.51} we obtain:
\begin{equation}\label{3.52}
\tau_j\leqslant C_2\omega(C_3\exp(-j)).
\end{equation}
So, estimates \eqref{3.47}, \eqref{3.49}, \eqref{3.51} are the results of the first circle of the computation and a starting point for the second circle. Similar to \eqref{3.35}, we define value $s_{j-1}$:
\begin{equation}\label{3.35*}
C(\nu)h_\omega(s_{j-1})^{-\frac2{p-1}-\nu}=\overline{K}_{j-1}^\theta, \quad\theta=(2e)^{-1},\ C(\nu)\text{ is from \eqref{3.33}}.
\end{equation}
Then the energy function $J_j(s_{j-1},\tau)$ satisfies the following differential inequality:
\begin{equation}\label{3.53}
J_j(s_{j-1},\tau)\leqslant \tilde{c}s_{j-1}\left(-\frac{d}{d\tau}J_j(s_{j-1},\tau)\right) +\overline{K}_{j-1}^\theta\quad \forall\,\tau\in(\delta+\tau_j,2^{-1}\rho_0)
\end{equation}
instead of \eqref{3.36}, and the following ''initial'' condition:
\begin{equation}\label{3.54}
J_j(s_{j-1},\delta+\tau_j)\leqslant\overline{K}_{j-1}.
\end{equation}
which is a consequence of inequality \eqref{3.47} from the first circle of the computations. Next we define $\tau_{j-1}$ by the analog of \eqref{3.37}:
\begin{equation}\label{3.37*}
J_j(s_{j-1},\delta+\tau_j+\tau_{j-1})=2\overline{K}_{j-1}^\theta.
\end{equation}
Similar to \eqref{3.38}, by \eqref{3.53}, \eqref{3.54}, \eqref{3.37*} we have the following relation:
\begin{equation}\label{3.38*}
J_j(s_{j-1},\tau)\leqslant2\tilde{c}s_{j-1}\left(-\frac{d}{d\tau}J_j(s_{j-1},\tau)\right) \quad\forall\,\tau\in(\delta+\tau_j,\delta+\tau_j+\tau_{j-1}).
\end{equation}
Solving this differential inequality by the ''initial'' condition \eqref{3.54}, we obtain:
\begin{equation}\label{3.39*}
J_j(s_{j-1},\tau)\leqslant\overline{K}_{j-1}\exp\left(-\frac{\tau-\delta-\tau_j}{2\tilde{c}s_{j-1}}\right) \quad\forall\,\tau\in(\delta+\tau_j,\delta+\tau_j+\tau_{j-1}).
\end{equation}
Definition \eqref{3.37*} of $\tau_{j-1}$ and estimate \eqref{3.39*} lead to the explicit estimate of $\tau_{j-1}$:
\begin{equation*}
\tau_{j-1}\leqslant 2\tilde{c}s_{j-1}\left(-\ln2+(1-\theta)\ln\overline{K}_{j-1}\right),
\end{equation*}
and, finally, to the following analogs of estimates \eqref{3.49}, \eqref{3.52}:
\begin{equation}\label{3.55}
s_{j-1}\leqslant C_3\exp(-j+1),\quad \tau_{j-1}\leqslant C_2\omega(C_3\exp(-j+1)).
\end{equation}
Inequality \eqref{3.31} yields the analog of \eqref{3.42}:
\begin{equation}\label{3.42*}
I_j(s_{j-1})\leqslant\frac{C_1(\nu)}{C(\nu)}
\left(C(\nu)h_\omega(s_{j-1})^{-\frac2{p-1}-\nu}\right).
\end{equation}
Summing estimates \eqref{3.37*} and \eqref{3.42*}, using definition \eqref{3.35*} of $s_{j-1}$ and keeping in mind the validity of property \eqref{3.45} for all $j\geqslant j_0$ with $j_0$ from \eqref{3.46*}, we get:
\begin{multline}\label{3.56}
\int_{\Omega^{\rho_0}(\delta+\tau_j+\tau_{j-1})}\left(\mid \nabla_x u_j\mid ^2+h_\omega(d(x))u_j^{p+1}\right)dx\leqslant \left(2+\frac{C_1(\nu)}{C(\nu)}\right)\overline{K}_{j-1}^\theta\leqslant \\ \leqslant\overline{K}_{j-2},\text{ if }j>j_0+1.
\end{multline}
This estimate is the result of the second circle of computations and a starting point for the next circle. Realizing $i$ such circles, we obtain the following analog of estimates \eqref{3.56}, \eqref{3.55}:
\begin{equation}\label{3.56*}
\int_{\Omega^{\rho_0}(\delta+\sum_{k=0}^{i-1}\tau_{j-k})}\left(\mid \nabla_x u_j\mid ^2+h_\omega(d(x))u_j^{p+1}\right)dx\leqslant \overline{K}_{j-i}.
\end{equation}
\begin{equation}\label{3.55*}
\tau_{j-k}\leqslant C_2\omega(C_3\exp(-j+k)),\quad
s_{j-k}\leqslant C_3\exp(-j+k),\quad \forall\,k\leqslant i-1.
\end{equation}
There are two restrictions on value $i$. First of them follows from \eqref{3.46}, \eqref{3.46*}:
\begin{equation}\label{3.57}
j-i\geqslant j_0:=1+\ln2+\ln\ln(2+C_1(\nu)C(\nu)^{-1}).
\end{equation}
The second restriction follows from the analog of \eqref{3.38}, \eqref{3.38*}, namely, estimates of the interval, where differential inequality for energy function $J_j(s_{j-i},\tau)$ has to be satisfied:
\begin{equation}\label{3.58}
\delta+\sum_{k=0}^{i-1}\tau_{j-k}\leqslant2^{-1}\rho_0.
\end{equation}
Due to estimate \eqref{3.55*} and monotonicity of $\omega(\cdot)$ we have:
\begin{equation}\label{3.59}
\begin{split}
\sum_{k=0}^{i-1}\tau_{j-k}\leqslant & C_2\sum_{k=0}^{i-1}\omega(C_3\exp(-j+k))\leqslant C_2C_3^{-1}\int_{C_3\exp(-j)}^{C_3\exp(-j+i)}r^{-1}\omega(r)dr\leqslant
\\ \leqslant &C_2C_3^{-1}\int_{0}^{C_3\exp(-(j-i))}r^{-1}\omega(r)dr=: C_2C_3^{-1}\Phi(j-i),
\end{split}
\end{equation}
where $\Phi(s)\to0$ as $s\to\infty$ due to the Dini condition \eqref{1.24}. Therefore for arbitrary $\delta>0$, $\rho_0>2\delta$ there exists finite $j^{(0)}=j^{(0)}(\delta,\rho_0)$ such that $\delta+C_2C_1^{-1}\Phi(j^{(0)})\leqslant2^{-1}\rho_0$ and, hence, condition \eqref{3.58} is satisfied if $j-i\geqslant j^{(0)}$. Thus, due to \eqref{3.59} estimate \eqref{3.56*} leads to:
\begin{equation}\label{3.60}
\int_{\Omega^{\rho_0}(\delta+C_2C_3^{-1}\Phi(j-i))} \left(\mid \nabla u_j\mid ^2+h_\omega(d(x))u_j^{p+1}\right)dx \leqslant \overline{K}_{j-i}\quad\forall\,j:j-i\geqslant\bar{j},
\end{equation}
where $\bar{j}:=\max\{j_0,j^{(0)},j'\}$, $j'$ is from \eqref{3.48}, $u_j(x)=u_{j,\delta}(x)$. Let us notice that for arbitrary small $\delta>0$ we can find finite number $j^{(1)}=j^{(1)}(\delta)$ such that:
\begin{equation}\label{3.61}
j^{(1)}=\min(j\in\mathbb{N}:\delta\geqslant K_j^{-(p-1)}).
\end{equation}
By definition \eqref{3.34} of $\overline{K}_j$ we have:
\begin{equation}\label{3.62}
\overline{K}_{j}\leqslant2\bar{c}K_j^{p+1}\quad\forall\,j\geqslant j^{(1)},
\end{equation}
which means that $\overline{K}_{j}$ does not depend on $\delta$ if $j\geqslant j^{(1)}$. Now by condition \eqref{1.24} we can find finite $j^{(2)}=j^{(2)}(\delta)\geqslant j^{(1)}(\delta)$, such that
\begin{equation*}
C_2C_3^{-1}\Phi(j-i)\leqslant\delta\quad\forall\,j:j-i>j^{(2)}
\end{equation*}
and, by \eqref{3.60} we get:
\begin{equation}\label{3.63}
\int_{\Omega^{\rho_0}(2\delta)} \left(\mid \nabla u_j\mid ^2+h_\omega(d(x))u_j^{p+1}\right)dx \leqslant \overline{K}_{j''}\quad\text{if }\,j>j'':=\max\{\bar{j},j^{(2)}\}.
\end{equation}
This estimate yields the following uniform with respect to $j\in\mathbb{N}$ a priori estimate:
\begin{equation}\label{3.64}
\mid u_{j,\delta}\mid _{H^1(\Omega^{\rho_0}(2\delta), \partial\Omega^{\rho_0}(2\delta)\cap\partial\Omega)}\leqslant C=C(\delta)<\infty\quad\forall\,j\in\mathbb{N},
\end{equation}
where for an arbitrary set $S\subset\partial\Omega$ by $H^1(\Omega,S)$ we define the closure in the norm of $H^1(\Omega)$ of set $C^1(\Omega,S):=\{f\in C^1(\Omega):f\mid _S=0\}$. Since $h_\omega(d(x))\geqslant0$ in $\overline{\Omega}$, all functions $u_{j,\delta}$ are subsolutions of the corresponding linear elliptic equation. Therefore, by Harnack inequality, for subsolutions of linear elliptic equations (see, e.g., \cite{GT}) we get
\begin{equation}\label{3.65}
\Big(\sup_{\Omega^{\rho_1}(3\delta)}u_{j,\delta}\Big)^2\leqslant c(\delta,\rho_0)\int_{\Omega^{\rho_0}(2\delta)} \mid u_{j,\delta}(x)\mid ^2dx\quad\forall\,j\in\mathbb{N},\ \forall\,\delta>0,\ \rho_1=\frac{\rho_0}2.
\end{equation}
By \eqref{3.65} and \eqref{3.63} we have:
\begin{equation}\label{3.66}
\sup_{\Omega^{\rho_1}(3\delta)}u_{j,\delta}
\leqslant c_1(\delta)\quad\forall\,j\in\mathbb{N},\ \forall\,\delta>0.
\end{equation}
In virtue of \eqref{3.12} the last inequality yields:
\begin{equation}\label{3.67}
\sup_{\Omega^{\rho_1}(3\delta)}u_{j}
\leqslant c_1(\delta)\quad\forall\,j\in\mathbb{N},\ \forall\,\delta>0,
\end{equation}
where $u_j$ is a solution to the problem \eqref{1.17}, \eqref{1.22}. It is easy to see that $u_j(x)$ is a solution to the following problem:
\begin{align}
  & -Lu_j=g_j(x):=H(x)u_j(x)^p\quad\text{in }\Omega^{\rho_1}(3\delta) \label{3.68}\\
  & u_j\mid _{\partial\Omega^{\rho_1}(3\delta)\cap\partial\Omega}=0 \quad\forall\,j\in\mathbb{N},\label{3.69}
\end{align}
where, $\mid g_j\mid _{L^q(\Omega^{\rho_1}(3\delta))}\leqslant c_2(\delta)$ $\forall\,j\in\mathbb{N}$, $\forall\,q>1$ due to \eqref{3.67}. Hence, by the classical $L^q$ a priori estimates for solutions of linear elliptic problems (see, for example, \cite{GT}) we get:
\begin{equation}\label{3.70}
\mid u_j\mid _{W^{2,q}(\Omega^{\rho_2}(4\delta))}\leqslant c_3(\delta)
\quad\forall\,j\in\mathbb{N},\ \forall\,q>1,\ \rho_2=\frac{\rho_1}2.
\end{equation}
By the comparison principle, sequence $\{u_j\}$, $j=1,2,...$, is monotonically nondecreasing in $\Omega$ and, hence, $u_j(x)\to u_\infty(x)$ pointwise for all $x\in\overline{\Omega}$. Then by the uniform estimate \eqref{3.70} and the compact embedding of the space $W^{2,q}(\Omega^{\rho_2}(4\delta))$ into $C^{1,\lambda}(\overline{\Omega}^{\rho_2}(4\delta))$, $0<\lambda<1-\frac Nq$, we have:
\begin{equation}\label{3.71}
\mid u_j-u_\infty\mid _{C^{1,\lambda}(\overline{\Omega}^{\rho_2}(4\delta))}\to0 \text{ as }j\to\infty.
\end{equation}
Since $\delta$ is an arbitrary positive number and $u_j=0$ on $\partial\Omega\setminus\{\cup_{i\leqslant l}\overline{G}_i\}$ it follows from \eqref{3.71} that $u_\infty(x)=0$ $\forall\,x\in\partial\Omega\setminus\{\cup_{i\leqslant l}\overline{G}_i\}$. Theorem~\ref{Th.2} is proved.
\end{proof}

%%%%%%%%%%%%%%%%%%%%%%%%%%%%%%%%%%%%%%%%%%%%%%%%%%%%%%%%%%%%%%%%%%%
\section{About the uniqueness of large solution}

Here we prove Theorem~\ref{Th.3}. Our proof consists in verifying that the equation under consideration has the so-called strong barrier property. This property was introduced in \cite{MV3,MV4}, where the sufficiency of this property for the uniqueness of the large solution was proved for the equation under consideration.

\begin{definition}\label{Def4.1} (see \cite{MV4}, Def. 2.6)
Let $z\in\partial\Omega$. We say that equation \eqref{1.17} possesses a strong barrier at point $z$ if there exists a number $R_z>0$ such that for every $r\in(0,R_z)$ there exists a positive supersolution $u=u_{r,z}\in C(\overline{\Omega}\cap B_r(z))$ of equation \eqref{1.17} in $\Omega\cap B_r(z)$, such that
\begin{equation}\label{4.1}
\lim_{y\to x,\ y\in\Omega\cap B_r(z)}u_{r,z}(y)=\infty \text{ for all }x\in\Omega\cap\partial B_r(z).
\end{equation}
If an equation has the mentioned property for an arbitrary point $z\in\partial\Omega$ then we say that the equation under consideration has the strong barrier property.
\end{definition}

Without loss of generality we suppose that $\Omega\subset\mathbb{R}_+^N$, $z=0\in\partial\Omega$, $B_R(0)\cap\Omega=\{x=(x'',x_N)\in B_R(0):x_N>0\}:=\Omega_{R}$ and $d(x):=dist(x,\partial\Omega)=x_N$ $\forall\,x\in\Omega_{R}$. Introduce also a weight function
\begin{equation}\label{4.2}
\rho(x):=dist(x,\partial B_R(0))=R-\mid x\mid =R-(\mid x''\mid ^2+x_N^2)^\frac12,
\end{equation}
and a surface $\Gamma_R:=\{x\in\Omega_{R}:\rho(x)=d(x)\}$. It is easy to check that $\Gamma_R$ is a paraboloid:
\begin{equation}\label{4.3}
\Gamma_R=\left\{x\in\Omega_{R}:x_N=\frac{R^2-\mid x''\mid ^2}{2R}\quad\forall x'':\mid x''\mid  \leqslant R\right\}.
\end{equation}
Now we are going to prove that equation \eqref{1.17} has the property from Definition~\ref{4.1} in the domain $\Omega\subset\mathbb{R}^N_+$. Namely, let $\{u_j(x)\}$ be an increasing sequence of solutions of equation \eqref{1.17} in the domain $\Omega_R$ satisfying the boundary conditions:
\begin{equation}\label{ad3}
u\mid _{\partial\Omega_{R}\cap\Omega}=K_j,\quad
u\mid _{\partial\Omega\cap B_R(0)}=0\quad
K_j\to\infty\text{ as }j\to\infty.
\end{equation}
We will prove that $u_\infty=\lim_{j\to\infty}u_j(x)$ is a strong barrier at point $0\in\partial\Omega$ for equation \eqref{1.17} in the sense of Definition~\ref{Def4.1}.
Firstly, we introduce a sequence $\{u_{j,\delta}(x)\}$ of solutions of equation \eqref{1.17}, satisfying the following ''regularized'' boundary conditions:
\begin{equation}\label{4.4}
\begin{split}
&u\mid _{\partial\Omega_{R}\cap\Omega}=K_j,\quad
u\mid _{\partial\Omega\cap B_R(0)}=K_j\xi_\delta(\mid x''\mid ),
\\&\text{where } \xi_\delta(s)=\begin{cases}
                0, & \forall\,s\leqslant R-\delta, \\
                1-\delta^{-1}(R-s), & \forall\,s>R-\delta
              \end{cases}
\end{split}
\end{equation}
and $\delta>0$ arbitrary small. By the comparison principle we have $u_{j,\delta}(x)\geqslant u_j(x)$ $\forall\,j\in\mathbb{N}$, $\forall\,\delta>0$ and, hence, $u_{\infty,\delta}(x)\geqslant u_\infty(x)$ $\forall\,x\in\overline{\Omega}_{R}$. The main part of our analysis consists in proving that
\begin{equation}\label{ad4}
u_{\infty,\delta}(x)=0\quad\forall\,x=(x'',x_N):x_N=0,\ \mid x''\mid <R-c\delta,
\end{equation}
where $c=const<\infty$ does not depend on $\delta$. This proof is some adaptation of the proof of Theorem~\ref{Th.2}. Similarly as in \eqref{3.1}, introduce families of subdomains:
\begin{align}\label{4.5}
&\Omega_{R,s}=\{x\in\Omega_{R}:\bar{\rho}(x)>s\} \quad\forall\,s:0<s<\frac R2,\\
&\Omega_{R}^s=\{x\in\Omega_{R}:0<\bar{\rho}(x)<s\} \quad\forall\,s:0<s<\frac R2,\label{4.6}
\end{align}
where $\bar{\rho}(x):=\min \{\rho(x),d(x)\}$, $d(x)=x_N$, $\rho(\cdot)$ is from \eqref{4.2}. If $u$ is an arbitrary nonegative solution of equation \eqref{1.17} in $\Omega_{R}$ then we introduce the following energy function:
\begin{equation}\label{4.7}
I(s):=\int_{\Omega_{R,s}}(\mid \nabla_xu\mid ^2+h_\omega(d(x))u^{p+1})dx \quad\forall\,s\in(0,2^{-1}R).
\end{equation}

\begin{lemma}\label{Lem4.1}
The function $I(\cdot)$ from \eqref{4.7} satisfies:
\begin{equation}\label{4.8}
I(s)\leqslant c_1\left[\int_0^sh_\omega(r)^\frac2{p+3} dr\right]^{-\frac{p+3}{p-1}} \quad\forall\,s\in(0,2^{-1}R),
\end{equation}
where constant $c_1<\infty$ does not depend on $u$.
\end{lemma}

\noindent The proof is similar to the proof of lemma~\ref{Lem3.1} with nonessential changes, so we omit it.

\begin{lemma}\label{Lem4.2}
Solution $u_j(x):=u_{j,\delta}$ of the problem \eqref{1.17}, \eqref{4.4} in the domain $\Omega_{R}$ satisfies the following a priori estimate:
\begin{equation}\label{4.9}
\int_{\Omega_{R}}(\mid \nabla u_j\mid ^2+h_\omega(d(x)) u_j^{p+1})dx\leqslant\overline{K}_j:=c_2(K_j^{p+1}+\delta^{-1}K_j^2),
\end{equation}
where $c_2<\infty$ does not depend on $j\in\mathbb{N}$.
\end{lemma}

\begin{proof}
It is easy to see that
\begin{equation*}
v_j(x):=u_j(x)-K_j\xi_\delta(\mid x\mid )=0\text{ on }\partial\Omega_{R}.
\end{equation*}
Now multiplying equation \eqref{1.17} by $v_j$ and integrating it by parts we get the analog of relation \eqref{3.14}:
\begin{multline}\label{4.10}
\int_{\Omega_{R}}\left(\sum_{i,k=1}^Na_{ik}u_{jx_i}u_{jx_k} +H(x)u_j^{p+1}\right)dx = K_j\int_{\Omega_{R}} \sum_{i,k=1}^Na_{ik}u_{jx_i}\xi_\delta(\mid x\mid )_{x_k}dx+\\
+K_j\int_{\Omega_{R}}H(x)u_j^q\xi_\delta(\mid x\mid )dx.
\end{multline}
Further proof coincides with the proof of lemma~\ref{Lem3.2}.
\end{proof}

Similarly as in \eqref{3.14*}, introduce family of subdomains of $\Omega^s_R$:
\begin{multline}\label{3.14**}
\Omega^s_R(\tau):=\Omega^s_R\cap\big\{(x'',x_N):\mid x''\mid <R-\tau,x_N<s\big\}\quad\forall\,\tau\in(\delta,R),\ \forall\,s\in(0,s_\delta),\\ s_\delta:=\delta-\frac{\delta^2}{2R}= \frac{R^2-(R-\delta)^2}{2R}\text{ (see \eqref{4.3})},
\end{multline}
and, similarly as in \eqref{3.15*}, introduce energy function
\begin{equation}\label{3.15**}
J_j(s,\tau):=\int_{\Omega_R^{2s}(\tau)}\left(\mid \nabla u_j\mid ^2+ h_\omega(x_N)u_j^{p+1}\right)\zeta_s(x_N)dx,
\end{equation}
where $\zeta_s(\cdot)$ is a function from \eqref{3.13*}, \eqref{3.15*}, $u_j=u_{j,\delta}$.

\begin{lemma}\label{Lem4.3}
The energy function $J_j(s,\tau)$ satisfies the following relation:
\begin{equation}\label{3.16**}
\begin{split}
&J_j(s,\tau)\leqslant c_3s\left(-\frac{d}{d\tau}J_j(s,\tau)\right)+Ch_\omega(s)^{-\frac2{p-1}-\nu}\quad \forall\,\tau\in(\delta,R),\ \forall\,j\in\mathbb{N},
\\ &\forall\,s\in(0,2^{-1}s_\delta),\ \forall\,\nu>0,\ C=C(\nu)\to\infty\text{ as }\nu\to0,
\end{split}
\end{equation}
where $c_3,C(\nu)$ do not depend on $j$; $s_\delta$ is from \eqref{3.14**}.
\end{lemma}

\begin{proof}
Multiplying equation \eqref{1.17} by $u_j(x)\zeta_s(x_N)$, where $u_j:=u_{j,\delta}$, and integrating it over $\Omega^{2s}_R(\tau)$, $\tau\in(\delta,R)$, $2s\leqslant s_\delta=\delta-\frac{\delta^2}{2R}$, we obtain the following analog of \eqref{3.17*}:
\begin{equation}\label{3.17**}
\begin{split}
&\overline{J}_j(s,\tau):=\int_{\Omega^{2s}_R(\tau)} \left(\sum_{i,k=1}^Na_{ik}(x)u_{jx_i}u_{jx_k} +H(x)u_j^{p+1}\right)\zeta_s(x_N)dx= \\&=\int_{\partial''\Omega^{2s}_R(\tau)} \sum_{i,k=1}^Na_{ik}(x)u_{jx_i}u_{j}\nu_k(x)\zeta_s(x_N)d\sigma -\\&-\int_{\Omega^{2s}_R(\tau)\setminus\Omega^s_R(\tau)} \sum_{i,k=1}^Na_{ik}u_{jx_i}\zeta_s(x_N)_{x_k}u_jdx:=R_1+R_2
\quad\forall\,\tau\in(\delta,R),\ s<2^{-1}s_\delta,
\end{split}
\end{equation}
where $\partial''\Omega^{2s}_R(\tau)=\{(x'',x_N):\mid x''\mid =R-\tau, 0<x_N<2s\}$. Now notice the following important property of subdomains:
\begin{equation}\label{ad1}
\Omega^{2s}_R(\tau)\setminus\Omega^{s}_R(\tau)\subset \Omega_{s}(\tau)\setminus\Omega_{2s}(\tau)\quad \forall\,\tau>\delta,\ \forall\,s:0<s<2^{-1}s_\delta
\end{equation}
and, hence,
\begin{equation}\label{ad2}
\int_{\Omega^{2s}_R(\tau)\setminus\Omega^{s}_R(\tau)} (\mid \nabla u_j\mid ^2+h_\omega(x_N)u_j^{p+1})dx\leqslant I_j(s)-I_j(2s) \quad\forall\,\tau>\delta,\ \forall\,s<2^{-1}s_\delta.
\end{equation}
Estimating the term $R_1$ in \eqref{3.17**} by the same way as in \eqref{3.18}--\eqref{3.24}, we obtain due to properties \eqref{ad1}, \eqref{ad2}:
\begin{equation}\label{3.24*}
\begin{split}
  &\mid R_1\mid \leqslant
  c_4s\int_{\partial''\Omega^{2s}_R(\tau)}\mid \nabla u_j\mid ^2\zeta_s(x_N)d\sigma+ \\ &+ s^{-1+\frac{p-1}{2(p+1)}}h_\omega(s)^{-\frac{1}{p+1}}\left(I_j(s)-I_j(2s)\right)^{1-\frac{p-1}{2(p+1)}}+
  \\&+s^{-1+\frac{p-1}{p+1}} h_\omega(s)^{-\frac{2}{p+1}} \left(I_j(s)-I_j(2s)\right)^{1-\frac{p-1}{p+1}}\quad \forall\,\tau\in(\delta,R),\ \forall\,s<\frac\delta2-\frac{\delta^2}{4R}.
\end{split}
\end{equation}
Analogously to \eqref{3.26}, we have:
\begin{equation}\label{3.26*}
\mid R_2\mid \leqslant cs^{-\left(1-\frac{p-1}{2(p+1)}\right)}h_\omega(s)^{-\frac1{p+1}}
(I_j(s)-I_j(2s))^{1-\frac{p-1}{2(p+1)}}\quad\forall\,s<2^{-1}s_\delta.
\end{equation}
Bearing in mind the following analog of relation \eqref{3.28}:
\begin{equation}\label{3.28*}
\int_{\partial''\Omega^{2s}(\tau)}\left(\mid \nabla u_j\mid ^2+h_\omega(x_N) u_j^{p+1}\right)\zeta_s(x_N) d\sigma\leqslant-c_0\frac{d}{d\tau}J_j(s,\tau),
\end{equation}
and using estimates \eqref{3.24*}, \eqref{3.26*} we get from \eqref{3.17**} inequality \eqref{3.29} for all $\tau\in(\delta,R)$ and $s\in(0,2^{-1}s_\delta)$. Estimating now $F_j(s)$ in the same way as in \eqref{3.30}--\eqref{3.32}, we obtain \eqref{3.33} and, consequently, necessary relation \eqref{3.16**}.
\end{proof}

Further proof of Theorem~\ref{Th.3} is similar to the corresponding part of the proof of Theorem~\ref{Th.2} with obvious changes. By the arguments, similar to \eqref{3.34}--\eqref{3.63}, we obtain the following analog of uniform estimate \eqref{3.63}:
\begin{equation}\label{3.63*}
\int_{\Omega^{s_\delta}(2\delta)}(\mid \nabla u_j\mid ^2+h_\omega(x_N)u_j^{p+1})dx\leqslant \overline{K}_{j''} \quad\forall\,j>j''=j''(\delta),
\end{equation}
where $j''(\delta)\to\infty$ as $\delta\to0$. Since $u_j(x):=u_{j,\delta}(x)=0$ $\forall\,x=(x'',x_N):x_N=0,\mid x''\mid <R-\delta$, we deduce from \eqref{3.63*} the following property by the same arguments as in \eqref{3.64}--\eqref{3.71}:
\begin{equation}\label{3.71*}
\mid u_{j,\delta}-u_{\infty,\delta}\mid _{C^{1,\lambda}(\overline{\Omega}^{2^{-1}s_\delta}(4\delta))}\to0 \text{ as }j\to\infty,
\end{equation}
and, consequently, we get property \eqref{ad4} with $c=4$. Now since $u_{\infty,\delta}(x)\geqslant u_\infty(x)$ $\forall\,x\in\overline{\Omega}_{R}$ $\forall\,\delta>0$ relation \eqref{3.71*} implies that $u_\infty(x)=0$ $\forall\,x=(x'',x_N):x_N=0,\mid x''\mid <R$. Moreover, $u_\infty(x)=\infty$ $\forall\,x=(x'',x_N):\mid x\mid =R, x_N>0$. Thus $u_\infty$ is the desired strong barrier for equation \eqref{1.17}. Theorem~\ref{Th.3} is proved.

%%%%%%%%%%%%%%%%%%%%%%%%%%%%%%%%%%%%%%%%%%%%%%%%%%%%%%%%%%%%%%%%%%%%%%%%%%%%%%%%%%%%%%%%%%
%%%%%%%%%%%%%%%%%%%%%%%%%%%%%%%%%%%%%%%%%%%%%%%%%%%%%%%%%%%%%%%%%%%%%%%%%%%%%%%%%%%%%%%%%%
\section{Necessity of the Dini condition for non-propagation of the point singularity}

In this section we prove Theorem~\ref{Th.1}.
%\begin{proof}
%{\bf Proof of Th.1.}
First of all, we construct a family of subsolutions $v_j(x)$, connected with solutions $u_j(x)$, $j=1,2, ...$ of the problem under consideration. Namely, introduce a family of subdomains $\Omega_j$ of the domain $\Omega$ from \eqref{1.1}:
\begin{equation}\label{2.1}
\Omega_j:=\biggr\{x\in\Omega=\mathbb{R}_+^N:\mid x'\mid ^2:=\sum_{i=2}^Nx_i^2<r_j^2,x_1\in\mathbb{R}^1\biggr\},\ r_j=2^{-j},\ j=1,2,...,
\end{equation}
and numbers
\begin{equation}\label{2.2}
a_j=\exp\left(-\mu(r_j)\right)=\exp\left(-\frac{\omega(r_j)}{r_j}\right), A_j=\left(a_jr_j^2\right)^\frac1{p-1}\quad\forall\,j\in\mathbb{N}.
\end{equation}
Consider now a family of auxiliary problems:
\begin{equation}\label{2.3}
-\Delta v+a_jv^p=0\text{ in }\Omega_j,
\end{equation}
\begin{equation}\label{2.4}
v=0\text{ on }\partial\Omega_j\setminus\{x:x_N=0\},
\end{equation}
\begin{equation}\label{2.5}
v\mid _{\partial\Omega_j\cap\{x_N=0\}}=K_j\delta_a(x),\quad
a\in L=\{x=(x_1,0,...,0)\}.
\end{equation}
Due to condition \eqref{1.9*} on $p$, inequality \eqref{1.6} is satisfied for equation \eqref{2.3}, and, hence, problem \eqref{2.3}--\eqref{2.5} has solution $v_j$, $j=1,2, ...$.
Since $u_j(x)\geqslant v_j(x)$ on $\partial\Omega_j$ then
\begin{equation}\label{2.6}
u_j(x)\geqslant v_j(x)\text{ in }\Omega_j\quad\forall\,j\in\mathbb{N}.
\end{equation}
Next, we estimate $v_j(x)$ from below. Let us perform the rescaling of the problem \eqref{2.3}--\eqref{2.5}. Introducing new variables and an unknown function:
\begin{equation}\label{2.7}
y=r_j^{-1}x,\ w_j(y)=A_jv_j(r_jy),\ y\in G:=\{y\in\mathbb{R}^N:\mid y'\mid <1,y_N>0\},
\end{equation}
where $A_j$ is from \eqref{2.2}, we obtain with respect to $w_j(y)$ the following problem:
\begin{equation}\label{2.8}
-\Delta w_j+w_j^p=0\text{ in }G,
\end{equation}
\begin{equation}\label{2.9}
w_j(y)=0\text{ on }\partial G\setminus\{y:y_N=0\},
\end{equation}
\begin{equation}\label{2.10}
w_j\mid _{\partial G\cap\{y_N=0\}}=K_jA_jr_j^{-(N-1)}\delta_{ar_j^{-1}}(y).
\end{equation}
It is easy to see that without loss of generality we can suppose that $a=0$. Let us specify now the choice of the sequence $\{K_j\}$:
\begin{equation}\label{2.11}
K_j:=A_j^{-1}r_j^{N-1}\ \forall\,j\in\mathbb{N}.
\end{equation}
Then $w_j(y)=w(y)$ $\forall\,j\in\mathbb{N}$, where $w(y)$ is a solution of the problem:
\begin{equation}\label{2.12}
-\Delta_yw+w^p=0\text{ in }G
\end{equation}
\begin{equation}\label{2.13}
w(y)=0\text{ on }\partial G\setminus\{y:y_N=0\}
\end{equation}
\begin{equation}\label{2.14}
w\mid _{\partial G\cap\{y_N=0\}}=\delta_0(y).
\end{equation}
It is obvious (due to comparison principle) that $w(\cdot)$ satisfies the estimate:
\begin{equation}\label{2.15}
0\leqslant w(y)\leqslant P_0(y,0)=
\frac{c_Ny_N}{(y_1^2+y_2^2+...+y_N^2)^\frac N2}
\quad\forall\,y\in G,
\end{equation}
where $P_0(\cdot,\cdot)$ is the Poisson kernel from \eqref{1.5}.
Therefore, function $w^{(y_1)}(y'):=w(y_1,y')$ has the following properties:
\begin{equation}\label{2.17}
\begin{aligned}
  &\mid w^{(y_1)}(\cdot)\mid _{L^\infty\left(B_{1,+}^{N-1}\right)}<\infty\ \forall\,y_1\neq0,
  \\ &\varphi_w(y_1):=\mid w^{(y_1)}(\cdot)\mid _{L^\infty\left(B_{1,+}^{N-1}\right)}\to0\text{ as }\mid y_1\mid \to\infty,
  \end{aligned}
\end{equation}
where $B_{1,+}^{N-1}=\left\{y'=(y_2,...,y_N)\in\mathbb{R}^{N-1}: \mid y'\mid <1,\ y_N>0\right\}$.
Let us fix an arbitrary value $y_1^{(0)}>0$ and consider solution $w$ as a solution of equation \eqref{2.12} in the infinite cylindrical domain $G\cap\left\{y:y_1>y_1^{(0)}\right\}$ satisfying boundary condition:
\begin{equation}\label{2.18}
w(y)=0\text{ on }\partial G\cap\left\{y:y_1>y_1^{(0)}\right\}.
\end{equation}
Due to lemma~3.1 from \cite{MSh1,MSh1-1}, for the solution $w$ there exists a number $\alpha>0$ such that
\begin{equation}\label{2.19}
\lim_{y_1\to\infty}\exp\left(\sqrt{\lambda_1}\left(y_1-y_1^{(0)}\right)\right)w(y) =\alpha\psi_1(y'),\quad\max_{y'\in B_{1,+}^{N-1}}\psi_1(y')=\psi_1(\tilde{y}')=1,
\end{equation}
uniformly in $B_{1,+}^{N-1}$. Here $\lambda_1>0$ is the first eigenvalue and $\psi_1$ is the corresponding normalized eigenfunction of $-\Delta$ in $B_{1,+}^{N-1}$, constant $\alpha=\alpha(y_1^{(0)})$ satisfies estimate:
\begin{equation}\label{2.20}
0<\alpha\leqslant c\quad\sup_{y'\in B_{1,+}^{N-1}} w^{(y_1^{(0)})}(y')=c\varphi_w(y_1^{(0)}),
\end{equation}
where $c<\infty$ doesn't depend on solution $w$, function $\varphi_w(\cdot)$ is from \eqref{2.17}.

\begin{remark}
Lemma 3.1 is proved for the cylindrical domain $G=\mathbb{R}^1\times B_1^{N-1}$ in \cite{MSh1,MSh1-1}. But its proof is based on the results of \S2 of the paper \cite{B1}, which are true for a much more general class of cylindrical domains, particularly, for $G=\mathbb{R}^1\times B_{1,+}^{N-1}$.
\end{remark}

Thus, due to \eqref{2.17}, \eqref{2.11}, \eqref{2.7}, it follows from \eqref{2.19} the existence of a constant $\beta:y_1^{(0)}<\beta<\infty$, which does not depend on $j$, such that
\begin{multline}\label{2.21}
\frac{\alpha}{2A_j}\psi_1\left(r_j^{-1}x'\right) \exp\biggr(-\sqrt{\lambda_1}\biggr(\frac{x_1}{r_j}-y_1^{(0)}\biggr)\biggr) \leqslant v_j(x)\leqslant \\ \leqslant\frac{2\alpha}{A_j}\psi_1\left(r_j^{-1}x'\right) \exp\biggr(-\sqrt{\lambda_1}\biggr(\frac{x_1}{r_j}-y_1^{(0)}\biggr)\biggr)
\quad\forall\,x\in\Omega_j:x_1>\beta r_j,\ \forall\,j\in\mathbb{N}.
\end{multline}
Due to \eqref{2.6} inequality \eqref{2.21} yields the first rough estimate from below of the solution $u_j$:
\begin{equation}\label{2.22}
u_j(x)\geqslant v_j(x)\geqslant B_j\psi_1\left(r_j^{-1}x'\right)\exp\left(-\sqrt{\lambda_1}\frac{x_1}{r_j}\right) \quad\forall\,x\in\Omega_j:x_1>\beta r_j,
\end{equation}
where $B_j=(2A_j)^{-1}\alpha(y_1^{(0)})\exp\left(\sqrt{\lambda_1}y_1^{(0)}\right)$. Let us define number $\tau_j>0$ by the following relation:
\begin{equation}\label{2.23} B_j\exp\left(-\sqrt{\lambda_1}\frac{\tau_j}{r_j}\right)=A_{j-1}^{-1} =\left(a_{j-1}r_{j-1}^2\right)^{-\frac{1}{p-1}},
\end{equation}
which yields by simple computations:
\begin{equation}\label{2.24}
\frac{\tau_j}{r_j}= \frac{\mu(r_j)-\mu(r_{j-1})}{\sqrt{\lambda_1}(p-1)}+c_1,\quad
c_1=y_1^{(0)}+\frac{\ln\alpha}{\sqrt{\lambda_1}}+\frac{(3-p)\ln2}{(p-1)\sqrt{\lambda_1}}.
\end{equation}
By condition \eqref{1.9}, we have $(\mu(r_j)-\mu(r_{j-1}))\to\infty$ as $j\to\infty$. Hence, there exists a number $j'$, such that
\begin{equation}\label{2.25}
\frac{r_j\left(\mu(r_j)-\mu(r_{j-1})\right)}{\sqrt{\lambda_1}(p-1)}
\leqslant\tau_j\leqslant \frac{2r_j\left(\mu(r_j)-\mu(r_{j-1})\right)}{\sqrt{\lambda_1}(p-1)}
\quad\forall\,j>j'
\end{equation}
and, additionally,
\begin{equation}\label{2.26}
\tau_j>\beta r_j\quad\forall\,j\geqslant j'.
\end{equation}
As a consequence of definition \eqref{2.23}, estimate \eqref{2.22} implies:
\begin{equation}\label{2.27}
u_j(\tau_j,x')\geqslant v_j(\tau_j,x')\geqslant A_{j-1}^{-1}\psi_1\left(r_j^{-1}x'\right)
\quad\forall\,j>j'.
\end{equation}
Let us fix arbitrary large $j>j'$ in \eqref{2.27} and consider a sequence $\left\{v_{j-k}(x)\right\}$, $k=1,2, ... ,j-j'$ ($j'$ is from \eqref{2.25}, \eqref{2.26}) of solutions of the following boundary value problems
\begin{equation}\label{2.28}
-\Delta v_{j-k}+a_{j-k}v_{j-k}^p=0\text{ in } \Omega_{j-k}\cap\{x_1>\tau_j+...+\tau_{j-k+1}\},
\end{equation}
\begin{equation}\label{2.29}
v_{j-k}=0\text{ on }\partial\Omega_{j-k}
\cap\biggr\{x_1>\sum_{i=0}^{k-1}\tau_{j-i}\biggr\},
\end{equation}
\begin{equation}\label{2.30}
v_{j-k}\biggr(\sum_{i=0}^{k-1}\tau_{j-i},x'\biggr)=\gamma_{j-k}(x') :=
\begin{cases}
A_{j-k}^{-1}\psi_1\left(\frac{x'}{r_{j-k+1}}\right)
&\text{if
}\mid x'\mid <r_{j-k+1},\\
0 &\text{if }r_{j-k+1}\leqslant\mid x'\mid \leqslant r_{j-k}.
\end{cases}
\end{equation}
where $\Omega_{j-k}=\left\{x\in\Omega:\mid x'\mid <r_{j-k}\right\}$, $A_{j-k}=\left(a_{j-k}r_{j-k}^2\right)^\frac1{p-1}$, $a_{j-k}=\exp(-\mu(r_{j-k}))$. To define the sequence $\{\tau_{j-k}\}$, $k=1,2,...$, we need to transform problem \eqref{2.28}-\eqref{2.30} to new variables:
\begin{equation}\label{2.31}
y=r_{j-k}^{-1}x,\ w_{j-k}(y):=A_{j-k}v_{j-k}(r_{j-k}y).
\end{equation}
It is easy to see that all these functions $w_{j-k}(y)$ can be obtained as a shift of the unique function $w(y)$:
\begin{equation}\label{2.31*}
w_{j-k}(y_1,y'):=w\biggr(y_1-\sum_{i=0}^{k-1}\tau_{j-i}r_{j-k}^{-1},y'\biggr),
\end{equation}
where $w(y)$ is a solution of the problem:
\begin{equation}\label{2.32}
-\Delta_yw+w^p=0\text{ in } G\cap\{y_1>0\},
\end{equation}
\begin{equation}\label{2.33}
w=0\text{ on }\partial G
\cap\left\{y_1>0\right\},
\end{equation}
\begin{equation}\label{2.34}
w(0,y'):=
\begin{cases}
\psi_1\left(2y'\right)
&\text{if
}\mid y'\mid <2^{-1},\\
0 &\text{if }2^{-1}\leqslant\mid y'\mid \leqslant1.
\end{cases}
\end{equation}
Due to lemma~3.1 from \cite{MSh1,MSh1-1} there exists a number $\alpha_1>0$, such that the function $w$ has the following property:
\begin{equation}\label{2.35}
\lim_{y_1\to\infty}\exp\left(\sqrt{\lambda_1}y_1\right)w(y)= \alpha_1\psi_1(y').
\end{equation}
Here constant $\alpha_1>0$ satisfies the estimate
\begin{equation*}
\alpha_1\leqslant c \sup_{y'\in B_{1,+}^{N-1}}\psi_1(2y')=c
\end{equation*}
with constant $c$, the same as in \eqref{2.20}. As in \eqref{2.21}, by definition \eqref{2.31*}, property \eqref{2.35} implies the existence of a constant $\beta_1<\infty$, which does not depend on $j$ and $k\leqslant j-1$, such that
\begin{multline}\label{2.36}
\frac{\alpha_1}{2A_{j-k}}\psi_1(r_{j-k}^{-1}x') \exp\left(\frac{-\sqrt{\lambda_1}(x_1-h_{j,k})}{r_{j-k}}\right)\leqslant v_{j-k}(x)\leqslant \frac{2\alpha_1}{A_{j-k}}\psi_1(r_{j-k}^{-1}x')\times \\ \times\exp\left(\frac{-\sqrt{\lambda_1}(x_1-h_{j,k})}{r_{j-k}}\right) \quad\forall\,x\in\Omega_{j-k}:x_1\geqslant h_{j,k}+r_{j-k}\beta_1; h_{j,k}:=\sum_{i=0}^{k-1}\tau_{j-i}.
\end{multline}
Let us define value $\tau_{j-k}$ by the relation:
\begin{equation}\label{2.37}
\frac{\alpha_1}{2A_{j-k}} \exp\left(-\sqrt{\lambda_1}\frac{\tau_{j-k}}{r_{j-k}}\right) =A_{j-k-1}^{-1},\quad k=1,2,... .
\end{equation}
Using the nonnegativity of $u_j$ in $\Omega$, properties \eqref{2.27} and boundary condition \eqref{2.30} with $k=1$, we get:
\begin{equation*}
u_j(\tau_j,x')\geqslant v_{j-1}(\tau_j,x')\quad \forall\,x':\mid x'\mid <r_j.
\end{equation*}
Using additionally that $v_{j-1}(\tau_j,x')=0$ if $r_j<\mid x'\mid <r_{j-1}$, we obtain
\begin{equation}\label{2.38}
u_j(\tau_j,x')\geqslant v_{j-1}(\tau_j,x')\quad \forall\,x':\mid x'\mid <r_{j-1}.
\end{equation}
Now using the comparison principle for solution $u_j$ and subsolution $v_{j-1}$ of equation \eqref{1.1} in the domain $\Omega_{j-1}\cap\{x_1>\tau_j\}$, we obtain
\begin{equation}\label{2.39}
u_j(x)\geqslant v_{j-1}(x)\quad \forall\,x\in\Omega_{j-1}\cap\{x_1>\tau_j\}.
\end{equation}
Next we will establish the main intermediate inequality:
\begin{equation}\label{2.40}
u_j(x)\geqslant v_{j-k}(x)\quad \forall\,x\in\Omega_{j-k}\cap \biggr\{x_1\geqslant\sum_{i=0}^{k-1}\tau_{j-i}\biggr\}\ \forall\,k:1\leqslant k<j-j'',
\end{equation}
where $j''<j$ does not depend on $j$. In virtue of \eqref{2.39} inequality \eqref{2.40} is true for $k=1$. Let us suppose that \eqref{2.40} holds for some $k\geqslant1$. Then we have to prove that
\begin{equation}\label{2.41}
u_j(x)\geqslant v_{j-k-1}(x)\quad \forall\,x\in\Omega_{j-k-1}\cap \Big\{x_1\geqslant\sum_{i=0}^{k}\tau_{j-i}\Big\}.
\end{equation}
To do this, it is sufficient, due to the comparison principle, to show that
\begin{equation}\label{2.42}
u_j\Big(\sum_{i=0}^{k}\tau_{j-i},x'\Big)\geqslant v_{j-k-1}\Big(\sum_{i=0}^{k}\tau_{j-i},x'\Big).
\end{equation}
From boundary condition \eqref{2.30} for the function $v_{j-k-1}(x)$ we have:
\begin{equation}\label{2.43}
v_{j-k-1}\Big(\sum_{i=0}^{k}\tau_{j-i},x'\Big)=
\begin{cases}
  A_{j-k-1}^{-1}\psi_1\left(\frac{x'}{r_{j-k}}\right), & \mbox{if } \mid x'\mid \leqslant r_{j-k} \\
  0, & \mbox{if }r_{j-k}<\mid x'\mid \leqslant r_{j-k-1}.
\end{cases}
\end{equation}
Now if number $\tau_{j-k}$, defined by relation \eqref{2.37}, satisfies additionally the following inequality:
\begin{equation}\label{2.44}
\tau_{j-k}\geqslant\beta_1r_{j-k}\text{ with }\beta_1\text{ from }\eqref{2.36},
\end{equation}
then relation \eqref{2.36} yields:
\begin{equation*}%\label{2.44}
v_{j-k}\Big(\sum_{i=0}^k\tau_{j-i},x'\Big)\geqslant \frac{\alpha_1}{2A_{j-k}}\psi_1(r_{j-k}^{-1}x') \exp\Big(-\frac{\sqrt{\lambda_1}\tau_{j-k}}{r_{j-k}}\Big).
\end{equation*}
In virtue of definition \eqref{2.37} of $\tau_{j-k}$ the last inequality leads to:
\begin{equation}\label{2.45}
v_{j-k}\Big(\sum_{i=0}^k\tau_{j-i},x'\Big)\geqslant A_{j-k-1}^{-1}\psi_1(r_{j-k}^{-1}x')\quad \forall\,x':\mid x'\mid \leqslant r_{j-k}.
\end{equation}
Since $u_j\left(\sum_{i=0}^k\tau_{j-i},x'\right)\geqslant0= v_{j-k-1}\left(\sum_{i=0}^k\tau_{j-i},x'\right)$ if $r_{j-k}\leqslant\mid x'\mid \leqslant r_{j-k-1}$, relations \eqref{2.43} and \eqref{2.45} lead to \eqref{2.42} under the comparison principle. Thus, inequality \eqref{2.40} is proved for all $k\geqslant1$, such that estimate \eqref{2.44} holds for $\tau_{j-k}$, defined by \eqref{2.37}. Using standard computations, we deduce from definition \eqref{2.37}:
\begin{equation*}
\frac{\tau_{j-k}}{r_{j-k}}=\frac{\mu(r_{j-k})-\mu(r_{j-k-1})} {\sqrt{\lambda_1}(p-1)}+c_2,\quad c_2=\frac{\ln\alpha_1}{\sqrt{\lambda_1}}+\frac{(3-p)\ln2}{(p-1)\sqrt{\lambda_1}}.
\end{equation*}
Due to condition \eqref{1.9} on the function $\mu(\cdot)$ it follows from the last relation that there exists a constant $\widetilde{j}<\infty$, which does not depend on $j$, such that the following inequalities hold:
\begin{multline}\label{2.46}
\frac{r_{j-k}\left(\mu(r_{j-k})-\mu(r_{j-k-1})\right)}{\sqrt{\lambda_1}(p-1)}
\leqslant\tau_{j-k}\leqslant \frac{2\left(\mu(r_{j-k})-\mu(r_{j-k-1})\right)r_{j-k}}{\sqrt{\lambda_1}(p-1)},
\\ \forall\,k<j-\widetilde{j}.
\end{multline}
Additionally, it follows from condition \eqref{1.9} the existence of a number $\widetilde{j}_1=\widetilde{j}_1(\beta_1)<\infty$, such that
\begin{equation}\label{Ad11}
\frac{\mu(r_{j-k})-\mu(r_{j-k-1})}{\sqrt{\lambda_1}(p-1)}\geqslant\beta_1\quad\forall\,j\in\mathbb{N},\ \forall\,k<j:j-k\geqslant\widetilde{j}_1.
\end{equation}
Therefore, the main intermediate inequality \eqref{2.40} holds with $j''=\max\{j',\widetilde{j},\widetilde{j}_1\}$.
Moreover, condition \eqref{1.9} yields the existence of a constant ${\ae}<1$ and a number $\widehat{j}<\infty$, such that
\begin{equation}\label{2.47}
\mu(r_{j-k-1})\mu(r_{j-k})^{-1}<{\ae}<1\quad\forall\,j\in\mathbb{N},\ \forall\,k<j:j-k>\widehat{j}.
\end{equation}
Therefore, it follows from \eqref{2.25}, \eqref{2.46} that
\begin{equation}\label{2.48}
\frac{2\mu(r_{j-k})}{\sqrt{\lambda_1}(p-1)} \geqslant\frac{\tau_{j-k}}{r_{j-k}}\geqslant \frac{(1-{\ae})\mu(r_{j-k})}{\sqrt{\lambda_1}(p-1)}\quad
\forall\,j\in\mathbb{N},\ \forall\,k<j:j-k>\max\{j'',\widehat{j}\}.
\end{equation}
Hence, by definition \eqref{1.4} of $\mu(\cdot)$, we have:
\begin{multline}\label{2.49}
\sum_{k=0}^{j-i}\tau_{j-k}\geqslant {\ae}_1\sum_{k=0}^{j-i}\omega(r_{j-k})\geqslant {\ae}_1\sum_{k=0}^{j-i}\int_{r_{j-k}}^{r_{j-k-1}}\frac{\omega(s)}sds= {\ae}_1\int_{r_{j}}^{r_{i-1}}\frac{\omega(s)}sds
\\ \forall\,j,i:j>i>j^{(1)}:=\max\{j'',\widehat{j}\}=\max\{j',\widetilde{j},\widetilde{j}_1,\widehat{j}\},\ {\ae}_1:=\frac{1-{\ae}}{\sqrt{\lambda_1}(p-1)}.
\end{multline}
Additionally, the left inequality in \eqref{2.48} and assumption \eqref{1.9} on the function $\omega(\cdot)$ imply that $\tau_i\to0$ as $i\to\infty$. Therefore, for an arbitrary fixed number $g>0$ there exists a number $i_g$, such that $\tau_i<g$ $\forall\,i>i_g$. Then, by virtue of condition \eqref{1.10} and inequalities \eqref{2.49}, there exists a number $\bar{j}=\bar{j}(i,g)<\infty$, such that
\begin{equation}\label{2.50}
g<\sum_{k=0}^{\overline{j}-i}\tau_{\overline{j}-k}=\sum_{k=0}^{\overline{j}-i}\tau_{i+k},\ \sum_{k=1}^{\overline{j}-i}\tau_{i+k}\leqslant g\quad\forall\,i>j^{(2)}:=\max\left\{j^{(1)},i_g\right\}.
\end{equation}
Notice that $\overline{j}(i,g)-i\to\infty$ as $i\to\infty$. Let us rewrite proved relation \eqref{2.40} in the following equivalent form:
\begin{equation}\label{2.40*}
u_j(x)\geqslant v_{i}(x)\quad \forall\,x=(x_1,x')\in\Omega_{i}\cap \Big\{x_1\geqslant\sum_{k=1}^{j-i}\tau_{i+k}\Big\}\ \forall\,j>i>j''.
\end{equation}
Let us define a sequence of points $\{x^{(i)}\}$:
\begin{equation*}
\begin{split}
  &x^{(i)}=\left(x^{(i)}_1,x^{'(i)}\right): x'^{(i)}=r_{i+1}\,\tilde{y}',\ \tilde{y}'\text{ is from \eqref{2.19}},\\
  &x_1^{(i)}=\sum_{k=0}^{\overline{j}-i}\tau_{i+k}=\sum_{k=0}^{\overline{j}-i}\tau_{\overline{j}-k}\quad \forall\,i>j^{(2)},\ \overline{j}=\overline{j}(i,g)\text{ is from \eqref{2.50}}.
\end{split}
\end{equation*}
We deduce from \eqref{2.30} after a simple transformation:
\begin{equation*}
v_i\Big(\sum_{k=1}^{\overline{j}-i}\tau_{i+k},x'\Big)=A_i^{-1} \psi_1(r_{i+1}^{-1}x')\quad\forall\,x':\mid x'\mid <r_{i+1}.
\end{equation*}
Therefore using \eqref{2.50} and the definition of point $x^{(i)}$ we get in virtue of \eqref{2.19}:
\begin{equation}\label{2.51}
v_i\left(g+\lambda_i\tau_{i},x'^{(i)}\right)=A_i^{-1} \psi_1(r_{i+1}^{-1}x'^{(i)})=A_i^{-1} \quad\forall\,i>j^{(2)},\ 0\leqslant\lambda_i<1.
\end{equation}
Let us define sequence $X^{(i)}=(X_1^{(i)},x'^{(i)})$, $X_1^{(i)}=g+\lambda_i\tau_i$. Then since $A_i^{-1}\to\infty$ as $i\to\infty$ we deduce from \eqref{2.40*} and \eqref{2.51}:
\begin{equation}\label{2.52}
u_\infty(X^{(i)})\geqslant u_{\overline{j}}(X^{(i)})\geqslant v_i(X^{(i)}) =A_i^{-1}\to\infty\text{ as }i\to\infty,
\end{equation}
where $\overline{j}=\overline{j}(i,g)$ is from \eqref{2.50}. Since $X_1^{(i)}\to g$ and $x'^{(i)}\to0$ as $i\to\infty$, relation \eqref{2.52} yields $u_\infty(g,0,...,0)=\infty$. Since $(g,0,...,0)$ is an arbitrary point from $L\cap\{\mathbb{R}_+^1\}=L_+$ then $u_\infty\mid _{L\cap\{\mathbb{R}_+^1\}}=\infty$. Now we return to the model problem \eqref{2.12}--\eqref{2.14} and consider its solution $w$ as a solution of equation \eqref{2.12} in the infinite cylindrical domain $G\cap\{y:y_1<y_1^{(0)}<0\}$, satisfying boundary condition: $w(y)=0$ on $\partial G\cap\{y:y_1<y_1^{(0)}<0\}$. If we repeat all above analysis using this solution $w$ in the mentioned cylindrical domain, we obtain that $u_\infty\mid _{L\cap\{\mathbb{R}_-^1\}}=\infty$. Theorem~\ref{Th.1} is proved.
%\end{proof}

\vskip3.5mm
{\bf Acknowledgements.} The research of A.~Shishkov for this paper has been supported by the RUDN University Strategic Academic Leadership Program.

\end{document}